\documentclass{article}
\usepackage{amsmath,amsfonts,amsthm,amssymb,comment,enumerate,amsxtra,url,tikz-cd,tcolorbox,graphicx,enumitem,bbm,biblatex}
\usepackage{xcolor} 
\colorlet{mdtRed}{red!50!black}
\definecolor{dblue}{rgb}{0,0,.6}
\usepackage[colorlinks]{hyperref}
\hypersetup{linkcolor=blue,citecolor=dblue,filecolor=dullmagenta,urlcolor=mdtRed}
\usepackage{amsthm}

\newtcolorbox{mymathbox}[1][]{colback=white, sharp corners, #1}

\newtheorem{theorem}[equation]{Theorem}
\newtheorem{corollary}[equation]{Corollary}
\newtheorem{lemma}[equation]{Lemma}
\newtheorem{proposition}[equation]{Proposition}

\theoremstyle{remark}

\numberwithin{equation}{section}

\newcommand{\e}[1]{\mathbb{E}\left[#1\right]}
\newcommand*\diff{\mathop{}\!\mathrm{d}}

\DeclareMathOperator{\Vv}{{\mathbb Var}}

\begin{document}

\title{Road layout in the KPZ class}

\author{M\'arton Bal\'azs\thanks{University of Bristol \href{mailto:m.balazs@bristol.ac.uk}{m.balazs@bristol.ac.uk}},
Sudeshna Bhattacharjee\thanks{Indian Institute of Science \href{mailto:sudeshnab@iisc.ac.in}{sudeshnab@iisc.ac.in}},\\
Karambir Das\thanks{Indian Institute of Science \href{mailto:karambirdas@iisc.ac.in}{karambirdas@iisc.ac.in}},
David Harper\thanks{Georgia Institute of Technology \href{mailto:dharper40@gatech.edu}{dharper40@gatech.edu}}}

\maketitle

\begin{abstract}
 We propose a road layout and traffic model, based on last passage percolation (LPP). An easy na\"ive argument shows that coalescence of traffic trajectories is essential to be considered when observing traffic networks around us. This is a fundamental feature in first passage percolation (FPP) models where nearby geodesics naturally coalesce in search of the easiest passage through the landscape. Road designers seek the same in pursuing cost savings, hence FPP geodesics are straightforward candidates to model road layouts. Unfortunately no detailed knowledge is rigorously available on FPP geodesics. To address this, we use exponential LPP instead to build a stochastic model of road traffic and prove certain characteristics thereof. Cars start from every point of the lattice and follow half-infinite geodesics in random directions. Exponential LPP is known to be in the KPZ universality class and it is widely expected that FPP shares very similar properties, hence our findings should equally apply to FPP-based modelling. We address several traffic-related quantities of this model and compare our theorems to real life road networks.
\end{abstract}

\section{Introduction}
Road networks and traffic patterns are clearly important phenomena in everyday life. Cars start from many locations and generally take random destinations to travel to. One's first thought might be to model car movements as independent straight trajectories. We show in Section \ref{sc:poi} that this na\"ive model results in divergent traffic densities in every region of the plane, which clearly does not align with observations.

Cars require roads to run on, and these cannot be built in arbitrary density and through arbitrary landscapes. Hence an interesting structure of road networks emerges, largely driven by geographic, historic, and social circumstances. The aim of this paper is to provide a mathematical model to describe some characteristics of such networks.

Characteristics of road networks have been considered in numerous works. As in this paper we are interested in a mathematical model to construct road networks, we only mention a handful of papers, first Barth\'elemy and Flammini \cite{bar_fla_mod_urban_st_patt}, who build urban networks by adding new centres and connecting them to the existing network in certain optimal way.

Aldous \cite{aldous_scaleinv_sp_nw} introduced the notion of \emph{Scale-invariant random spatial networks} as a potential model for road networks. This was followed by a natural construction of such an object by Kendall \cite{kendall_rnd_lines_ms} and by Kahn \cite{kahn_impr_poi_lines} based on a Poisson line process. We will base our model on a directed KPZ-style construction, hence we do not expect it to be scale-invariant.

Molinero and Hernando \cite{mol_her_model_4_gen_rd} introduce an algorithm based on a succession of Delanuay triangulation of a dense network of the centres of interest plus auxiliary points, and only keeping paths that are shortest between the original nodes of interest in the resulting network.

We build a model where, rather than the locations to connect, the environment surrounding the roads has a crucial effect on the geometry. Points to connect are thought of as having a homogeneous configuration throughout the plane, whereas obstacles roads need to get around will be random which in turn causes nontrivial network characteristics.

The environment we live in presents geographical challenges to road building. If we model this in the simplest possible way with i.i.d.\ cost distribution on some lattice, and try to build roads that minimise overall cost, then we naturally arrive to first passage percolation (FPP) models.

FPP was originally proposed as a model for the flow of fluid through porous media by Hammersley and Welsh in 1965 \cite{HW65}. Since then, FPP has evolved into a major area of interest in probability theory. See \cite{ADH15} for an extensive account on significant results in this field. Despite this active interest, many significant unanswered questions remain.

One of the fundamental objects in FPP are \emph{geodesics}, paths that collect minimal costs connecting two points among a random penalty landscape on the plane. These will be our candidates to model roads. In fact \cite{HW65} already asked about ``highways and byways'' in terms of edges frequented by geodesics, and the question received a partial answer by Ahlberg, Hanson and Hoffman \cite{ahl_hans_hoff_no_geo_fpp}.

For quantitative analysis of our road model we need refined estimates on geodesics, which unfortunately are not yet available for FPP models. Instead, we switch to exponential last passage percolation (LPP), where extensive results on last passage times and geometry of geodesics are now available. FPP is widely believed to belong to the \emph{KPZ universality class} -- a property proved for LPP. This implies that the geometry of geodesics in the two models should share basic characteristics that we use throughout our arguments. Thus, switching to LPP from FPP seems a reasonable move. Of the vast literature on the KPZ universality class, we refer to the surveys by Corwin \cite{corwin_kpz_univ16}, Ferrari-Spohn \cite{fer-spohn_rnd_growth11}, and Quastel \cite{quas_intro_kpz12}.

LPP models have been subject of intensive research in the last decades. Instead of minimizing weights (which we called costs so far), geodesics are maximising them under the constraint that paths can only take up or right steps. We define this model precisely in Section \ref{Model Definitions and Notations}.

The connection between roads and models in the KPZ universality class is not new. Solon, Bunin, Chu and Kardar \cite{sol_bun_chu_kar_opti_paths_poly} compared optimal paths in road networks with directed polymers in random medium (DPRM). LPP can be considered as extremal case of DPRM, and shares similar scaling exponents. Solon et al.\ find that shortest paths on the road network can be well approximated with DPRM -- albeit with the environment exhibiting a power law rather than in exponential LPP. They also conclude that long-range correlations in the road network play an important role in the scaling properties of optimal paths.

To model cars, we augment the basic layer of LPP models with a network of half-infinite geodesics, the a.s.\ existence of which was established by Ferrari-Pimentel and Coupier \cite{C11,FP05}. To be more precise, for any given direction and starting point, a.s.\ there exists a unique half-infinite geodesic from that point going into the given asymptotic direction. These form a perfect model of a car pursuing a distant destination in a given direction, while following the geodesic road network given to it. Each lattice point of \(\mathbb Z^2\) is thought of as the starting point of a car, and each car picks an independent random direction for its a.s.\ geodesic. As \(\mathbb Z^2\) is countable, the cars can jointly follow their own randomly oriented geodesics on a probability one event. In FPP models the direction could be Uniform(\(0,\,2\pi)\), in our LPP setup this must be restricted to Uniform(\(0,\,\frac\pi2)\). In fact we will assume some \(\varepsilon\) separation from the trivial angles \(0\) and \(\frac\pi2\).

We then discuss the following questions; notice that the model is translation-invariant, hence our inquiry for the origin is not restricting generality:
\begin{enumerate}
    \item Probabilistic estimates for the furthest distance a car can come from to the origin (see Section \ref{bounds for the furthest distance a car can come from}).
    \item Probabilistic bounds for the number of cars passing through a fixed point (see Section \ref{Bounds for Number of Cars}).
    \item How far one needs to go from a point to see high-traffic roads i.e., geodesics used by several cars at the same time (see Section \ref{Distance to find a road with large number of cars})?
\end{enumerate}
The purpose of this paper is to introduce this theoretical traffic model based on LPP, and to rigorously derive some of its interesting characteristics. We do not claim a good fit with reality: in Section \ref{sc:elev}, we compare our findings to real-world traffic data and point to various reasons why our model stays in the theoretical domain rather than proving useful for practical applications about road networks.

\subsection{A Poisson model for road network}\label{sc:poi}
Here we make the na\"ive assumption that cars go in a straight line i.e., the environment has no effect on their trajectories. The aim of this part is to demonstrate the need of a more elaborate model: this simple Poisson model cannot properly describe road networks.

Consider a homogeneous Poisson point process on $\mathbb R^2$ with intensity 1. These Poisson points represent the starting points of cars. A car chooses a direction uniformly between $0$ to $2\pi$ independently of everything, and it travels along a straight line segment of length $\ell$ in this direction. The length \(\ell\) of the trip follows the Exponential$(\gamma)$ distribution with a positive fixed parameter \(\gamma\), and is independent of all the other variables and cars.

This process can be thought of as a marked point process $\xi$ where the starting point $X_i$ of car \(i\) is a point in the homogeneous Poisson point process on \(\mathbb R^2\), and this gets decorated with the mark $(\theta_i,\,\ell_i)$ which represent the direction the car goes towards and the distance it travels respectively. The mark space is $
[0,2\pi )\times[0,\infty)
$, equipped with the standard product Borel \(\sigma\)-algebra. By the marking theroem \cite[Theorem 5.6]{last2017lectures}, $\xi$ is a Poisson point process on $\bigl(\mathbb{R}^2\times [0,2\pi)\times[0,\infty)\bigr)$. We are after the number \(N_r\) of cars that come \(r\) close to the origin.

Without loss of generality we now assume that a car starts from coordinate \((-z,0)\) of \(\mathbb R^2\), where \(z\) is a positive real. If \(z\le r\) then the car already starts inside the disc \(B(0,\,r)\) of radius \(r\) around the origin.

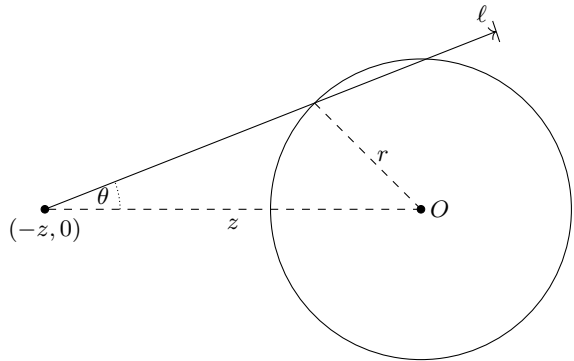
\begin{figure}[!ht]
 \begin{center}
  \begin{tikzpicture}
   \draw(0,0)circle(2)node[right]{\small\(O\)};
   \filldraw(0,0)circle(0.05);
   \draw[dashed](0,0)--(-1.414,1.414)node[midway,right]{\small\(r\)};
   \filldraw(-5,0)circle(0.05)node[below]{\small\((-z,0)\)};
   \draw[dashed](0,0)--(-5,0)node[midway,below]{\small\(z\)};
   \draw[->](-5,0)--(1,2.366)node[above left]{\small\(\ell\)};
   \draw(0.953,2.506)--(1.047,2.226);
   \draw[densely dotted](-4,0)arc(0:21.5:1)node[midway,left]{\small\(\theta\)};
  \end{tikzpicture}
 \end{center}
 \caption{A car starting from \((-z,0)\) and intersecting the disc \(B(0,\,r)\) or radius \(r\) around the origin.}\label{fig:carr}
\end{figure}

For \(z>r\), as seen from Figure \ref{fig:carr} and a bit of trigonometry, the car intersects \(B(0,\,r)\) if and only if \(|\sin\theta|\le\tfrac rz\) and \(\ell\ge z\cos\theta-\sqrt{r^2-z^2\sin^2\theta}\). For a car in general position the quantity \(z\) is to be replaced by the distance of the car's starting point from the origin. This way, exactly cars of the marked point process \(\xi\) in the above region of \(\bigl(\mathbb{R}^2\times [0,2\pi)\times[0,\infty)\bigr)\) will make it into \(B(0,\,r)\). Hence the number \(N_r\) of such cars is Poisson distributed with parameter equal the mean of this number.

Assuming a car is starting at distance \(z>r\) from the origin, we estimate its probability to hit \(B(0,\,r)\) by
\[
 \mathbb P\bigl\{|\sin\theta|\le\tfrac rz\cap\ell\ge z\cos\theta-\sqrt{r^2-z^2\sin^2\theta}\bigr\}\ge\mathbb P\bigl\{|\theta|\le\tfrac rz\cap\ell\ge z\bigr\}=\frac r{\pi z}\cdot e^{-\gamma z}.
\]
Breaking up the homogeneous Poisson process on \(\mathbb R^2\) w.r.t.\ polar coordinates, we get from here
\[
 \e{N_r}\ge\int_0^{2\pi}\int_0^r1\cdot z\diff z\diff\varphi+\int_0^{2\pi}\int_r^\infty\frac r{\pi z}\cdot e^{-\gamma z}\cdot z\diff z\diff\varphi=\pi r^2+\frac{2r}\gamma\cdot e^{-\gamma r}.
\]
We conclude that the number of cars intersecting the disc is Poisson with mean at least the right-hand side of this display.

The mean distance travelled by cars is \(\frac1\gamma\), and since we already fixed the Poisson intensity of cars at 1, we can think of this as a large number. (For actual figures, if the UK's 33.5m cars formed a homogeneous Poisson process on the UK's 243,610\,\(\text{km}^2\) area, then they would be spaced an average of 43\,metres apart. The typical driving distance is around 10\dots15\,km, which is far longer. Hence \(\frac1\gamma\) can be assumed to be much greater than \(\frac12\), which is the average distance between Poisson points of intensity 1, using the units of this article. Of course spatial density fluctuates drasticly between urban and rural areas.) Similarly, the radius \(r\) of interest e.g., for a homeowner wishing for a quiet house, can be considered \(\mathcal O(1)\). Hence we find that even in a moderate sized garden, plenty of cars should pass.  Taking the driving distance to infinity (\(\gamma\searrow0\)) makes \(N_r\) divergent.

This is not what we find in real life: the coalescence of the cars' trajectories is a significant missing feature in this na\"ive model.

\section{Model definitions, notations and results}\label{Model Definitions and Notations}
We first define the exponential last passage percolation model on $\mathbb{Z}^2$. We assign i.i.d.\ random variables $\{\tau_{v}\}_{v \in \mathbb{Z}^2}$ to each vertex of $\mathbb{Z}^2$, where $\tau_{v}$'s are distributed as Exp(1). Let, $u,v \in \mathbb{Z}^2$ be such that $u \leq v$ (i.e., if $u=(u_1,u_2),v=(v_1,v_2)$ then $u_1 \leq v_1$ and $u_2 \leq v_2$).
For an up-right path $\gamma$ between $u$ and $v$ we define $ \ell (\gamma)$ to be $\sum_{v \in \gamma \setminus \{u,v\}} \tau_{v}.$ Let $T(u,v):=$max$\{ \ell(\gamma): \gamma$ is an up-right path from $u$ to $v$\}. $T(u,v)$ is the \textit{last passage time} between $u$ and $v$. Clearly, as the number of up-right paths between $u$ and $v$ is finite, the maximum is always attained. Between any two points $u,v \in \mathbb{Z}^2$ maximum attaining paths are called \textit{geodesics}. As $\tau_{v}$ has a continuous distribution, between any two points $u\le v \in \mathbb{Z}^2$ almost surely there exists a unique geodesic. As a consequence of this uniqueness geodesics do not form loops. Hence, geodesics starting from ordered vertices (ordered in the $x+y=0$ direction) always stay ordered. This property of geodesics is known as \textit{planarity} and will be used in our proofs. The a.s.\ unique geodesic will be denoted by $\Gamma_{u,v}$. An infinite up-right path $u_0,u_1,u_2...$ in $\mathbb{Z}^2$ is called a \textit{semi-infinite geodesic starting from $u_0$}, if every finite segment of it is a geodesic. A semi-infinite geodesic starting from $u \in \mathbb{Z}^2$ is said to have \textit{direction} $\theta$ if $\lim_{n \rightarrow \infty} \frac{u_n}{|u_n|}$ exists and equals $\theta$ (we will always identify an element on the unit circle with its corresponding angle) and is denoted by $\Gamma_u^{\theta}$. It is a fact that for a fixed direction $\theta$, almost surely starting from any point $u \in \mathbb{Z}^2$ there exists a unique semi-infinite geodesic at the direction $\theta$ \cite{C11,FP05}. 
Now we will explain the kind of traffic problems we are interested in. Let $\varepsilon>0$ be fixed and imagine cars starting from each vertex of $\mathbb{Z}^2$ and picking up uniform direction from the interval $(\varepsilon,\pi/2-\varepsilon)$ independent of each other, also independent of the vertex weights. The cars then travel via the a.s.\ unique semi-infinite geodesic in the chosen direction from the starting point. Notice that we simply assume that cars travel infinitely far. While in the na\"ive Poisson model they travelled much further than the radius \(r\) of the neighborhood we considered, infinite driving distances would have made that model obsolete. As we shall see, this is not the case for our LPP traffic model.

See Figure \ref{fig:simu} for a simulation of some cars in the model. The main phenomenon we see is the empty regions towards the middle of the picture: most points around there will not see any cars passing from the distance to the bottom of the picture. When cars are started from all of \(\mathbb Z^2\), that translates into a limited close neighbourhood from where cars can arrive to a point say, the origin. All other cars from further away will have coalesced into busy roads somewhere else that avoid passing through the origin.
\begin{figure}[!ht]
 \includegraphics[width=\textwidth]{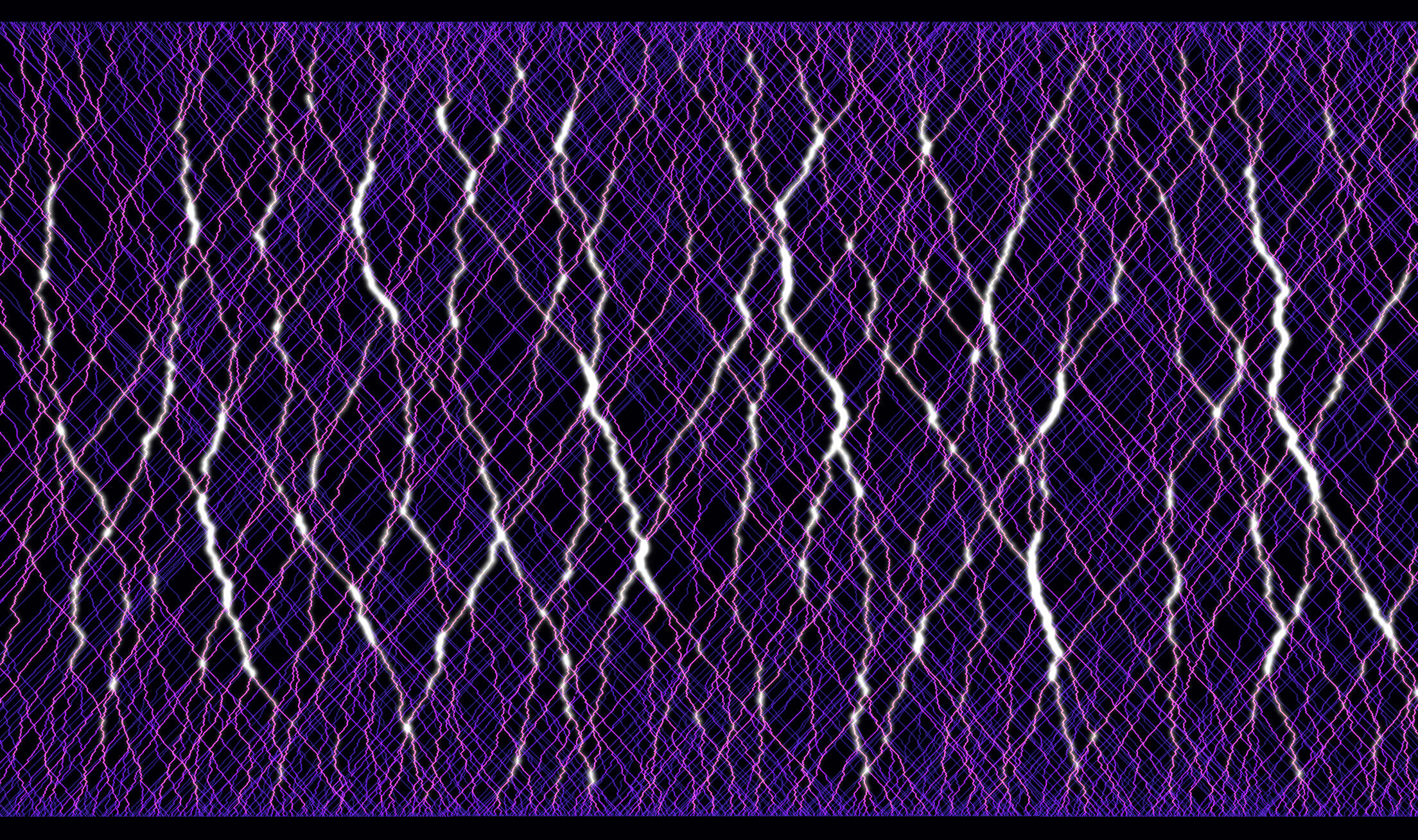}
 \caption{Simulation of geodesics in random directions, the picture is rotated left by \(45^\text o\). Each point on the bottom line picks another one randomly on the top within the admissible angle range \(0^\text o\,\text-\,90^\text o\) and sends a geodesic there (i.e., this simulation only launches cars from one line, as opposed to our model where this happens from each point of \(\mathbb Z^2\)). Many geodesics coalesce for most of their travel, forming the heavy traffic paths which are marked in lighter colors. These leave gaps towards the middle of the picture where no geodesics pass. The dark mesh of low-traffic geodesics close to \(0^\text o\) and \(90^\text o\) demonstrate the need of separation from the edge of the angle domain: geodesics with such angles have little room to wiggle, hence do not coalesce much.}\label{fig:simu} 
\end{figure}

Of the several questions that can be asked, we will address the tail behaviour of the following quantities:
\begin{itemize}
 \item How many cars pass through the origin (Theorems \ref{finiteness_theorem}, \ref{number of cars upper bound})?
 \item What is the farthest distance a car arrives from to the origin (Theorem \ref{depth_bounds_theorem})?
 \item How far do we have to walk to find a busy road that carries certain volume of traffic (Theorems \ref{high_probability_event}, \ref{road_lower_bound})?
\end{itemize}

It is this last question that seemed tractable from the observations point of view. In Section \ref{sc:elev} we compare our result to UK road layouts and traffic statistics, and we discuss the limitations of this experiment.

\subsection{Further notations and results}
Let $\varepsilon>0$ be fixed; from now on all the constants we will obtain from the theorems will depend only on this $\varepsilon$. From each vertex $v \in \mathbb{Z}^2$ an angle is chosen according to i.i.d.\ uniform $(\varepsilon, \frac{\pi}{2}-\varepsilon)$ distribution and we consider the a.s.\ unique semi-infinite geodesic in the chosen direction. We will denote the angle chosen by $v$ by $\theta_v$ and the semi-infinite geodesic in that direction by $\Gamma_v^{\theta_v}$. For $r \in \mathbb{Z}$ let $\boldsymbol{r}$ denote the vertex $(r,r) \in \mathbb{Z}^2$. For convenience sometimes we will work with the rotated axes $x+y=0$ and $x-y=0$. They will be called space axis and time axis respectively. For a vertex $v \in \mathbb{Z}^2$, $\phi(v)$ will denote the time coordinate of $v$ and $\psi(v)$ will denote the space coordinate of $v$. Precisely, for $u,v \in \mathbb{Z}^2$, we have 
\begin{displaymath}
    \phi(u,v)=u+v,\qquad\psi(u,v)=u-v.
\end{displaymath}For $u,v \in \mathbb{Z}^2$, $\Gamma_{u,v}$ denotes the geodesic between $u$ and $v$. For a deterministic direction $\alpha, \Gamma_v^{\alpha}$ denotes the a.s.\ semi-infinite geodesic starting from $v$ in the direction $\alpha$. The line $x+y=T$ will be denoted by $\mathcal{L}_T$. For a geodesic $\Gamma, \Gamma(T)$ will denote the random intersection point of $\Gamma$ and $\mathcal{L}_T.$
\\
Let $N$ denote the number of cars passing through the origin. More precisely we define, 
\begin{displaymath}
N:=\#\{u \in \mathbb{Z}^2: \boldsymbol{0} \in \Gamma_u^{\theta_u}\}.
\end{displaymath}
We have the following theorem.
\begin{theorem}
\label{finiteness_theorem}
$N$ is finite almost surely.
\end{theorem}
\begin{proof} This will follow from Theorem \ref{number of cars upper bound} which we will prove in Section \ref{Bounds for Number of Cars}.
\end{proof}
But $N$ has infinite expectation. To show this we first need the following definition. For $n \in \mathbb{N}$, define $N_n$ to be the number of cars starting from the line $\mathcal{L}_{-n}$ and going through the origin. We have the following theorem.
\begin{theorem}
\label{expectation_1 theorem}
$\mathbb{E}(N_n)=1$.
\end{theorem}
\begin{proof}
For $x \in \mathcal{L}_{-n},$ let $\mathcal{G}_x$ denote the event that the car starting from $x$ goes through the origin. Then 
\begin{displaymath}
    N_n=\sum_{x \in \mathcal{L}_{-n}}\mathbbm{1}_{\mathcal{G}_x}.
\end{displaymath}
So,
\begin{displaymath}
    \mathbb{E}(N_n)=\sum_{x \in \mathcal{L}_{-n}} \mathbb{P}(\mathcal{G}_x).
\end{displaymath}
For $x \in \mathcal{L}_{-n}$, we define $\mathcal{G}_{-x}$ to be the event that the car starting from $\boldsymbol{0}$ goes through $-x$. We have 
\begin{displaymath}
    \mathbb{P}(\mathcal{G}_x)=\mathbb{P}(\mathcal{G}_{-x}).
\end{displaymath}
So,
\begin{displaymath}
     \mathbb{E}(N_n)=\sum_{x \in \mathcal{L}_{-n}} \mathbb{E}(\mathbbm{1}_{\mathcal{G}_{-x}})=1.
\end{displaymath}
\end{proof}
The following corollary is not surprising then, and is also in line with the divergence of the na\"ive Poisson model.
\begin{corollary}
\label{infinite_expectation}
$\mathbb{E}(N)=\infty$.
\end{corollary}
\begin{proof}
This is clear using Theorem \ref{expectation_1 theorem}. Indeed we have,
\begin{displaymath}
    \mathbb{E}(N)=\sum_{n=1}^{\infty}\mathbb{E}(N_n)=\infty.
\end{displaymath}
\end{proof}

We now dive a bit deeper into the scaling properties of the model.
\begin{theorem}
\label{number of cars upper bound}
There exist constants $C_1, c_1>0$ (depending on $\varepsilon$) and $n$ sufficiently large such that 
\begin{displaymath}
    \frac{c_1}{n^{1/4}} \leq \mathbb{P}(N \geq n) \leq \frac{C_1 \log n}{n^{1/4}}.
\end{displaymath}
\end{theorem}

To get estimates for the furthest distance a car can come from we define the following random variable.
\begin{displaymath}
    D:=\max\{n \in \mathbb{N}: \exists u \in \mathcal{L}_{-n} \text{ such that } \boldsymbol{0} \in \Gamma_u^{\theta_u} \}.
\end{displaymath}
\begin{theorem}
\label{depth_bounds_theorem}
There exists $C,c >0$ (depending on $\varepsilon$) such that the following holds
\begin{displaymath}
cn^{-1/3} \leq \mathbb{P}(D \geq n) \leq Cn^{-1/3}.
\end{displaymath}
\end{theorem}

Finally, to investigate the distance to the nearest busy road, we define
\begin{displaymath}
    T_n:=\min \{ |\psi(v)| : v \in \mathcal{L}_0, N_v \geq n^{4/3} \}.
\end{displaymath}
We have the following theorem which is direct consequence of Theorem \ref{number of cars upper bound}.
\begin{theorem}
\label{high_probability_event}
    For $\delta>0$ sufficiently small (depending on $\varepsilon$) we have that there exists constant $C>0$ (depending on $\varepsilon$) and  for sufficiently large $n$ (depending on $\varepsilon$)
    \begin{displaymath}
        \mathbb{P}\left (T_n \geq \frac{ \delta n^{1/3}}{\log n} \right) \geq 1-C\delta.
    \end{displaymath}
    In particular, this implies there exists $c$ (depending on $\varepsilon$) such that
    \[
    \mathbb{E}\left(T_n \right) \geq \frac{cn^{1/3}}{\log n}.
    \]
\end{theorem}
\begin{proof}The proof follows from a union bound. Let $V$ denote the line segment with $\lceil{\frac{\delta n^{1/3}}{ \log n}}\rceil$ many vertices and with midpoint $\boldsymbol{0}$ on the line $\mathcal{L}_0$. Then 
\begin{displaymath}
    \left \{T_n \leq \frac{\delta n^{1/3}}{ \log n} \right \} \subset \bigcup_{v \in V} \left \{N_v \geq n^{4/3} \right \}.
\end{displaymath}
Hence, by Theorem \ref{number of cars upper bound} upper bound we have for some constant $C>0$ (depending on $\varepsilon$)
\begin{displaymath}
    \mathbb{P}\left (T_n \leq \frac{\delta n^{1/3}}{\log n} \right) \leq C \delta.
\end{displaymath}
Choosing $C \delta <1$ we get Theorem \ref{high_probability_event}.
\end{proof}
The above theorem says that, with high probability, to see a vertex having $n^{4/3}$ many cars through it, we need to go at least sufficiently small multiple of $\frac{n^{1/3}}{\log n}$ distance from origin. The following theorem shows that, with non-vanishing probability, there is at least one vertex within $n^{1/3}$ distance around the origin that has $n^{4/3}$ many cars through it. Precisely, for a fixed constant $\ell_0$ we make a slight modification to the definition of $T_n$ and call the new variable $T^{\ell_0}_n$.
\begin{displaymath}
 T^{\ell_0}_n:=\min \left \{ |\psi(v)| : v \in \mathcal{L}_0, N_v \geq \frac{n^{4/3}}{\ell_0} \right \}.
\end{displaymath}
We then have the following theorem. 
\begin{theorem}
\label{road_lower_bound}
There exists $c>0, \ell_0$ (depending on $\varepsilon$) such that for sufficiently large $n$ (depending on $\varepsilon$) we have
\begin{displaymath}
 \mathbb{P}(T^{\ell_0}_n\leq n^{1/3}) \geq c.
\end{displaymath}
\end{theorem}
We remark that a constant could also be factored in front of \(n^{1/3}\) inside the probability. For brevity, we do not carry this out throughout our proofs.

The rest of the paper is devoted to proving these theorems and to compare the latter two to actual road traffic statistics.

\section{Bounds for the furthest distance a car can come from}
Recall that \begin{displaymath}
    D:=\max\{n \in \mathbb{N}: \exists u \in \mathcal{L}_{-n} \text{ such that } \boldsymbol{0} \in \Gamma_u^{\theta_u} \}.
\end{displaymath}
\label{bounds for the furthest distance a car can come from}
\subsection{Proof of Theorem \ref{depth_bounds_theorem} upper bound} First we divide the interval $(\varepsilon, \frac{\pi}{2}-\varepsilon)$ into disjoint sub-intervals. Let $c_n=\frac{n^{1/3}}{\lfloor n^{1/3} \rfloor}$. Then $\{c_n\}$ is a uniformly bounded sequence. Now, for $1 \leq i \leq \lfloor n^{1/3} \rfloor$ we partition the interval $(\varepsilon, \frac{\pi}{2}-\varepsilon$) into intervals $A_i$ each of equal length $\frac{c_n(\frac{\pi}{2}-2 \varepsilon)}{n^{1/3}}$.  For $1 \leq i \leq \lfloor n^{1/3} \rfloor$ we define the following random variables.
\begin{displaymath}
 D_i:=\max\{m \in \mathbb{N}: \exists u \in \mathcal{L}_{-m} \text{ such that } \theta_u \in A_i \text{ and } \boldsymbol{0} \in \Gamma_u^{\theta_u} \}.
\end{displaymath}
Clearly, we have
\begin{displaymath}
    \{D \geq n\} \subset \bigcup_{i=1}^{\lfloor n^{1/3} \rfloor} \{D_i \geq n\}.
\end{displaymath}
So, we have 
\begin{equation}
\label{union_upper_bound}
    \mathbb{P}(D \geq n) \leq \sum_{i=1}^{\lfloor n^{1/3} \rfloor} \mathbb{P}(D_i \geq n).
\end{equation}
We fix $1 \leq i \leq \lfloor n^{1/3} \rfloor.$ We want to find an upper bound for $\mathbb{P}(D_i \geq n).$ We will apply an averaging argument. Let $V$ denote the line segment with $\lceil n^{2/3} \rceil$ vertices with $\mathcal{L}_0$ with midpoint $\boldsymbol{0}.$ For $v \in V$ we define the following random variables. 
\begin{displaymath}
 D_i^v:=\max\{m \in \mathbb{N}: \exists u \in \mathcal{L}_{-n} \text{ such that } \theta_u \in A_i \text{ and } v \in \Gamma_u^{\theta_u} \}.
\end{displaymath}
Clearly by the translation invariance of the model, for all $v \in V$ we have 
\begin{displaymath}
    \mathbb{P}(D_i \geq n)=\mathbb{P}(D_i^v \geq n).
\end{displaymath}
Hence,
\begin{equation}
\label{depth_averaging}
    \lceil n^{2/3} \rceil \mathbb{P}(D_i \geq n) = \sum_{v \in V}\mathbb{P}(D_i^v \geq n)=\mathbb{E}(\widehat{D_i}), 
\end{equation}
where $\widehat{D_i}$ is defined as follows.
\begin{displaymath}
    \widehat{D_i}:=\sum_{v \in V}\mathbbm{1}_{\{D_i^v \geq n\}}.
\end{displaymath}
We will show that for sufficiently large $\ell, \mathbb{P}(\widehat{D_i} \geq \ell )$ has a stretched exponential in $\ell$ upper bound. We prove it precisely now. Let $ \alpha_i, \beta_i, \theta_i$ denote the left end point, right end point and midpoint of $A_i$ respectively and $n_i' \in \mathcal{L}_{-n}$ be the intersection point of $\mathcal{L}_{-n}$ and the line $y=(\tan \theta_i) x$. Consider the line segment $\widetilde{V}$ with $\lfloor \ell^{1/96}n^{2/3} \rfloor$ vertices on $\mathcal{L}_{-n}$ with midpoint $n_i'$. Let $v_1,v_2$ be the end points of $\widetilde{V}$ with $\psi(v_1) \leq \psi(v_2).$ Consider $\Gamma_{v_1}^{\alpha_i}$ and $\Gamma_{v_2}^{\beta_i}$ (see Figure \ref{fig: depth_upper_bound_figure}).
\begin{figure}[!ht]
    \includegraphics[width=13 cm]{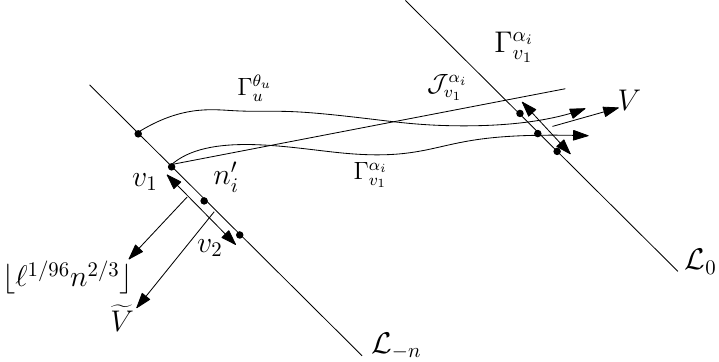}
    \caption{To prove the upper bound in Theorem \ref{depth_bounds_theorem} we fix $1 \leq i \leq \lfloor n^{1/3} \rfloor$. The event that there are more than $\ell$ many $u \in \mathcal{L}_{-n}$ with $\theta_u \in A_i$ and $\Gamma_u^{\theta_u}$ intersecting $V$ can happen in two ways. Either the above $u$ lies in $\mathcal{L}_{-n} \setminus \widetilde{V}$. For sufficiently large $\ell$ this will imply large transversal fluctuation for  either $\Gamma_{v_1}^{\alpha_i}$ or $\Gamma_{v_2}^{\beta_i}$ (the case that $\Gamma_{v_1}^{\alpha_i}$ can have large transversal fluctuation is shown in the figure. If for $u \in \mathcal{L}_{-n} \setminus \widetilde{V}$ as in the figure $, \Gamma_u^{\theta_u}$ intersects $V$ then by planarity of geodesics $\Gamma^{\alpha_i}_{v_i}$ will have large transversal fluctuation from the line $\mathcal{J}^{\alpha_i}_{v_1}$). Proposition \ref{transversal_fluctuation_of_semi_infinite_geodesic} gives a stretched exponential upper bound for this event. The other possibility is there are more than $\ell$ distinct vertices inside on $V$ that carries a geodesic from $\widetilde{V}$. Proposition \ref{coalescence_theorem} gives a stretched exponential upper bound for this event.}
    \label{fig: depth_upper_bound_figure} 
\end{figure}
We define the following event. 
\begin{itemize}
    \item $\mathcal{A}_i:=\{ \exists u \in \mathcal{L}_{-n} \setminus \widetilde{V}$ such that $\theta_u \in A_i$ and $v \in \Gamma_u^{\theta_u}$ for some $v \in V$\}.
\end{itemize}
We have 
\begin{displaymath}
    \mathbb{P}(\widehat{D_i} \geq \ell) \leq \mathbb{P}(\mathcal{A}_i)+\mathbb{P}(\{\widehat{D_i} \geq \ell\} \cap \mathcal{A}_i^c).
\end{displaymath}
Let $\mathcal{J}_{v_1}^{\alpha_i}$ (resp.\ $\mathcal{J}_{v_2}^{\beta_i}$) be the straight line starting from $v_1$ (resp.\ $v_2$) in the direction $\alpha_i$ (resp.\ $\beta_i$). Now, by planarity the event $\mathcal{A}_i$ implies either $\Gamma_{v_1}^{\alpha_i}$ or $\Gamma_{v_2}^{\beta_i}$ will have large traversal fluctuation on $\mathcal{L}_0$ from the lines $\mathcal{J}_{v_1}^{\alpha_i}$ and $\mathcal{J}_{v_2}^{\beta_i}$ respectively. (see Figure \ref{fig: depth_upper_bound_figure}). Since, $(\beta_i-\alpha_i)$ is of order $\frac{1}{n^{1/3}}$, for sufficiently large $\ell$ (depending on $\varepsilon$), the event $\mathcal{A}_i$ will imply that either $\Gamma_{v_1}^{\alpha_i}$ or $\Gamma_{v_2}^{\beta_i}$ will have transversal fluctuation larger than $\frac{\ell^{1/96}n^{2/3}}{4}$. More precisely let $\mathcal{J}_{v_1}^{\alpha_i}(0)$ (resp.\ $\mathcal{J}_{v_2}^{\beta_i}(0)$) denote the intersection point of the corresponding lines with $\mathcal{L}_0$. Similarly let $\Gamma^{\alpha_i}_{v_1}(0)$ (resp.\ $\Gamma^{\beta_i}_{v_2}(0)$) denote the intersection points of the corresponding geodesics with $\mathcal{L}_0.$ Then
\[
\mathcal{A}_i \subset \left \{\Gamma^{\alpha_i}_{v_1}(0)-\mathcal{J}_{v_1}^{\alpha_i}(0) \geq \frac{\ell^{1/96}n^{2/3}}{4} \right \} \bigcup \left \{ \mathcal{J}_{v_2}^{\beta_i}(0)-\Gamma^{\beta_i}_{v_2}(0)\geq \frac{\ell^{1/96}n^{2/3}}{4}\right\}.
\]

By Proposition \ref{transversal_fluctuation_of_semi_infinite_geodesic} we have upper bounds for the events on the right side. For sufficiently large $\ell$ we have 
\begin{displaymath}
    \mathbb{P}(\mathcal{A}_i) \leq Ce^{-c \ell^{3/96}}.
\end{displaymath}
Now, we consider the event $\{\widehat{D_i} \geq \ell\} \cap \mathcal{A}_i^c$. On this event there are at least $\ell$ distinct vertices on $V$ that has a geodesic coming from $\widetilde{V}.$
Using Proposition \ref{coalescence_theorem} we have for sufficiently large $\ell$ and $n$
\begin{displaymath}
    \mathbb{P}(\{\widehat{D_i} \geq \ell\} \cap \mathcal{A}_i^c) \leq Ce^{c \ell^{1/384}}.
\end{displaymath}
Combining all the above we have for $\ell<n^{0.01},n$ sufficiently large we have 
\begin{displaymath}
    \mathbb{P}(\widehat{D_i} \geq \ell) \leq Ce^{-c\ell^{1/384}}.
\end{displaymath}
So, finally we have that there exists $\ell_0$ (depending only on $\varepsilon$) and $\widetilde{C}$ (depending only on $\varepsilon$) such that for all large $n$
\begin{equation}
\label{expectation_calculation}
    \mathbb{E}\left(\widehat{D_i}\right) \leq \sum_{\ell \geq 1}\mathbb{P}(\widehat{D_i} \geq \ell) \leq  \ell_0 + \sum_{\ell_0 \leq \ell < \frac{n^{0.01}}{2}} C e^{-c \ell^{1/384}}+ \sum_{\frac{n^{0.01}}{2} \leq \ell < \lceil n^{2/3} \rceil}Ce^{-c'n^{1/38400}}< \widetilde{C}.
\end{equation}
So, from \eqref{depth_averaging} we have that there exists $C>0$ (depending only on $\varepsilon$) such that for all $1 \leq i \leq \lfloor n^{1/3} \rfloor$ we have
\begin{displaymath}
    \mathbb{P}(D_i \geq n) \leq \frac{C}{n^{2/3}}.
\end{displaymath}
So, from \eqref{union_upper_bound} there exists $C>0$ such that for large $n$
\begin{displaymath}
    \mathbb{P}(D \geq n) \leq \frac{C}{n^{1/3}}.
\end{displaymath}
This proves the upper bound in Theorem \ref{depth_bounds_theorem}. \qed\\
\subsection{Proof of Theorem \ref{depth_bounds_theorem} lower bound} We again consider disjoint sub-intervals of $(\varepsilon, \frac{\pi}{2}-\varepsilon)$. First we fix large constant $M$ which will be chosen later. For $1 \leq i \leq [\frac{n^{1/3}}{2M}]$, we consider the intervals $A_i \subset (\varepsilon, \frac{\pi}{2}-\varepsilon)$, each of length $\frac{(\frac{\pi}{2}-2 \varepsilon)}{n^{1/3}}$ and midpoint $\frac{Mi(\frac{\pi}{2}-2\varepsilon)}{n^{1/3}}$. Note that, unlike the previous case we are not partitioning the interval $(\varepsilon, \frac{\pi}{2}-\varepsilon$) this time. Instead we are taking points which are $\frac{M(\frac{\pi}{2}-2\varepsilon)}{n^{1/3}}$ distance apart from each other and we are taking intervals of length $\frac{(\frac{\pi}{2}-2 \varepsilon)}{n^{1/3}}$ around them. We do this so that the constants that we get at the end depend only on $\varepsilon$ but not on $M$. This will be elaborated later. Corresponding to each $A_i$ we define the following random variables:
\begin{displaymath}
 D_i:=\max \left \{m \in \mathbb{N}: \exists u \in \mathcal{L}_{-n} \text{ such that } \theta_u \in A_i \text{ and } \boldsymbol{0} \in \Gamma_u^{\theta_u} \right\}.
\end{displaymath}
Note that $D_i$ is defined exactly in the same way as we defined in the proof of the upper bound. The only difference is that the definition of the intervals $A_i$ are different now.
Clearly,
\begin{displaymath}
    \bigcup_{i=1}^{[\frac{n^{1/3}}{2M}]} \{ D_i \geq n\} \subset \{D \geq n\}.
\end{displaymath}
So, applying inclusion-exclusion principle and considering the first two terms (see Proposition \ref{pr:bonfe}) we have 
\begin{equation}
\label{inclusion_exclusion_for_depth}
  \mathbb{P}(D \geq n) \geq \sum_{i=1}^{[\frac{n^{1/3}}{2M}]}\mathbb{P}\left (\left \{D_i \geq n \right\}\right)-\sum_{1 \leq i <j \leq [\frac{n^{1/3}}{2M}]} \mathbb{P}\left (\left \{ D_i \geq n \right \} \cap \left \{D_j \geq n \right \} \right).
\end{equation}
For the first sum we do the following. We fix $1 \leq i \leq [\frac{n^{1/3}}{2M}]$. We consider the line segment $V_M$ with $\lceil 4Mn^{2/3} \rceil$ many vertices on $\mathcal{L}_0$ with midpoint $\boldsymbol{0}$ (see Figure \ref{fig: depth_lower_bound_figure}). For $v \in V_M$, we define 
\begin{displaymath}
 D_i^v:=\max \left \{m \in \mathbb{N}: \exists u \in \mathcal{L}_{-n} \text{ such that } \theta_u \in A_i \text{ and } v \in \Gamma_u^{\theta_u} \right \}.
\end{displaymath}
\begin{figure}[!ht]
    \includegraphics[width=13 cm]{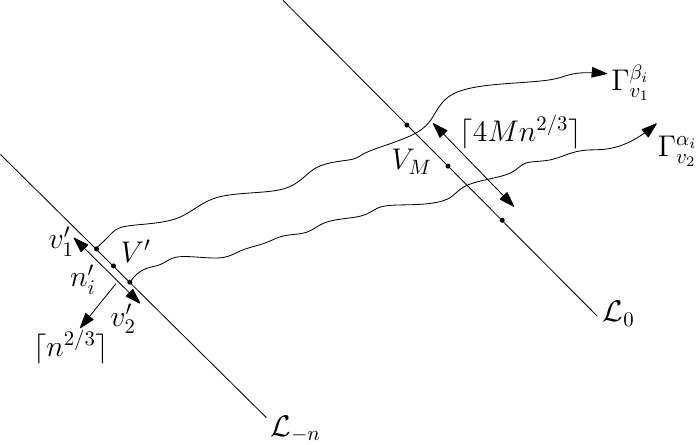}
    \caption{To prove the lower bound for the first sum in \eqref{inclusion_exclusion_for_depth} we fix an $i$. Using Proposition \ref{transversal_fluctuation_of_semi_infinite_geodesic} we can choose $M$ large enough so that on a positive probability event $\Gamma_{v_1'}^{\beta_i}$ and $\Gamma_{v_2'}^{\alpha_i}$ intersect $\mathcal{L}_0$ on $V_M$. Further, we consider an independent event on which there is at least one vertex on $V'$ that chooses an angle in $A_i$. Due to central limit theorem the later event is a positive probability event. Hence, the intersection of these two events is a positive probability event and on this intersection $\widehat{D}_i \geq 1$.}
    \label{fig: depth_lower_bound_figure} 
\end{figure}
So, same as before we have
\begin{equation}
\label{averaging_for_lower_bound}
    \lceil 4Mn^{2/3} \rceil \mathbb{P}(D_i \geq n)=\sum_{v \in V_M}\mathbb{P}(D_i^v \geq n)=\mathbb{E}(\widehat{D_i}),
\end{equation}
where $\widehat{D_i}$ is defined as 
\begin{displaymath}
    \widehat{D_i}:=\sum_{v \in V_M}\mathbbm{1}_{\{D_i^v \geq n\}}.
\end{displaymath}
We will show on a positive probability event $\widehat{D_i} \geq 1$. We construct this event now. Consider $V'$, a line segment with $\lceil n^{2/3} \rceil$ many vertices with midpoint $n_i'$ on $\mathcal{L}_{-n}$ and let $v_1',v_2'$ denote the end points of $V'$ with $\psi(v_1') \leq \psi(v_2').$ Consider the geodesics $\Gamma_{v_1'}^{\beta_i}$ and $\Gamma_{v_2'}^{\alpha_i}$.
We define the following event.
\begin{displaymath}
    \mathcal{C}_i:=\left \{ \Gamma_{v_1'}^{\beta_i}(0) \in V_M \text{ and } \Gamma_{v_2'}^{\alpha_i}(0) \in V_M \right \}.
\end{displaymath}
Again as $(\beta_i-\alpha_i)$ is of order $\frac{1}{n^{1/3}},$ the complement of the event $\mathcal{C}_i$ will imply that either $\Gamma_{v_1'}^{\beta_i}$ or $\Gamma_{v_2'}^{\alpha_i}$ will have transversal fluctuation more than $\frac{M}{2}n^{2/3}$. So, using Proposition \ref{transversal_fluctuation_of_semi_infinite_geodesic} we can choose $M$ large enough so that 
\begin{displaymath}
    \mathbb{P}(\mathcal{C}_i) \geq 0.99.
\end{displaymath}
Further, we consider the following event.
\begin{displaymath}
    \mathcal{D}_i:=\{\exists u \in V' \text{ such that } \theta_u \in A_i\}=\left \{\sum_{u \in V'}\mathbbm{1}_{\{\theta_u \in A_i\}} \geq 1 \right \}.
\end{displaymath}
So, we consider the random variable $X_i:=\sum_{u \in V'}\mathbbm{1}_{\{\theta_u \in A_i\}}.$ We have $\mathbb{E}(X_i)=\frac{\lceil n^{2/3} \rceil}{ n^{1/3}}$ and $\Vv(X_i)=(\frac{1}{n^{1/3}}-\frac{1}{n^{2/3}})\lceil n^{2/3} \rceil.$ By the Berry-Esseen inequality (Proposition \ref{p: Berry-Essen}) we have for all $n$ there exists a constant $c>0$ such that
\begin{displaymath}
    \left \vert\mathbb{P}(X_i-\mathbb{E}(X_i)\geq 0)-\frac{1}{2} \right \vert \leq \frac{c}{(n^{1/3}-1)^{3/2}n^{1/3}}.
\end{displaymath}
In particular, note that in this case the i.i.d. random variables in Proposition \ref{p: Berry-Essen} are $\mathbbm{1}_{\{\theta_u \in A_i\}}$ and we take $x$ in the statement of Proposition \ref{p: Berry-Essen} to be $0.$
Hence, for sufficiently large $n$ we have 
\begin{displaymath}
    \mathbb{P}(X_n-\mathbb{E}(X_n)\geq 0) \geq \frac{1}{4}.
\end{displaymath}
So, for sufficiently large $n$ we have
\begin{displaymath}
    \mathbb{P}(\mathcal{D}_i) \geq \frac{1}{4}.
\end{displaymath}
Finally, note that the events $\mathcal{C}_i$ and $\mathcal{D}_i$ are independent. So,
\begin{displaymath}
    \mathbb{P}(\mathcal{C}_i \cap \mathcal{D}_i) \geq \frac{0.99}{4}.
\end{displaymath}
Clearly, by planarity of geodesics (see Section \ref{Model Definitions and Notations} for definition), on the event $\mathcal{C}_i \cap \mathcal{D}_i, \widehat{D_i} \geq 1.$ So, from \eqref{averaging_for_lower_bound} we have that there exists $c_1>0$ such that
\begin{equation}
\label{first_sum_lower_bound_depth}
    \mathbb{P}(D_i \geq n) \geq \frac{c_1}{Mn^{2/3}}.
\end{equation}
Now we consider the second sum in \eqref{inclusion_exclusion_for_depth}. We fix $1\leq i <j \leq [\frac{n^{1/3}}{2M}],$  both large enough. For $1 \leq i \leq [\frac{n^{1/3}}{2M}]$ we define the following events.
\begin{displaymath}
    \mathcal{E}_i:=\{ \exists u \in \mathbb{Z}^2, \gamma \in A_i \text{ such that } \phi(u)=-n, \text{ and } \boldsymbol{0} \in \Gamma_u^{\gamma}\}.
\end{displaymath}
Observe that for all $1 \leq i \leq [\frac{n^{1/3}}{2M}]$
\begin{displaymath}
    \{D_i \geq n\} \subset \mathcal{E}_i.
\end{displaymath}
From now on we will work with $\mathcal{E}_i$'s. 
Consider the points $n_i',n_j' \in \mathcal{L}_{-n}$ defined as before. We have $|n_i'-n_j'|$ is of order $M|i-j|n^{2/3}.$ For simplicity let $k_{i,j}=M|i-j|$. We apply again an averaging argument. We fix $i$ and $j$. Same as before we define $\theta_i$ and $\theta_j$ to be the midpoints of the intervals $A_i$ and $A_j$. We define parallelograms $W_{i,j}$ depending on $\theta_i$ and $\theta_j$ as follows. Let both $\theta_i>\frac \pi4$ and $\theta_j >\frac \pi4$ and $\theta_i>\theta_j$. Then consider the parallelogram $W_{i,j}$ whose two sides lie on the lines $\mathcal{L}_{-\lceil\frac{n}{100k_{i,j}}\rceil}$ and $\mathcal{L}_{\lceil\frac{n}{100k_{i,j}}\rceil}$ each of length $n^{2/3}$ with midpoints $v_1', v_2'$ where $v_1'$ and $v_2'$ are the intersection points of the line $y=\tan(\theta_j)x$ with $\mathcal{L}_{-\lceil\frac{n}{100k_{i,j}}\rceil}$ and $\mathcal{L}_{\lceil\frac{n}{100k_{i,j}}\rceil}$ respectively. Now if $\theta_i<\frac \pi4$ and $\theta_j <\frac \pi4$ and $\theta_i>\theta_j$ then define the parallelogram $W_{i,j}$ exactly same as above with the only difference that the midpoints $v_1', v_2'$ are the intersection points of the line $y=\tan(\theta_i)x$ with $\mathcal{L}_{-\lceil\frac{n}{100k_{i,j}}\rceil}$ and $\mathcal{L}_{\lceil\frac{n}{100k_{i,j}}\rceil}$ respectively. Finally, if $\theta_i<\frac \pi 4<\theta_j$ then define $W_{i,j}$ exactly same as above with the midpoints $v_1', v_2'$ are the intersection points of the line $y=x$ with $\mathcal{L}_{-\lceil\frac{n}{100k_{i,j}}\rceil}$ and $\mathcal{L}_{\lceil\frac{n}{100k_{i,j}}\rceil}$ respectively (see Figures \ref{fig: second_sum_figure}, \ref{fig: second_sum_figure_1}).
Then for all $w \in W_{i,j}$ define the following events.
\begin{displaymath}
    \mathcal{E}_i^w:=\{ \exists u \in \mathbb{Z}^2 , \gamma \in A_i \text{ such that } \phi(u)=-n+\phi(w), \text{ and } w \in \Gamma_u^{\gamma}\}.
\end{displaymath}
Similarly we define $\mathcal{E}^w_j$.
We have for all $w \in W_{i,j}$
\begin{displaymath}
    \mathbb{P}(\mathcal{E}_i \cap \mathcal{E}_j)=\mathbb{P}(\mathcal{E}_i^w \cap \mathcal{E}_j^w).
\end{displaymath}
Hence,
\begin{equation}
\label{averaging}
    \frac{cn^{5/3}}{100k_{i,j}}\mathbb{P}(\mathcal{E}_i \cap \mathcal{E}_j)\le\sum_{w \in W}\mathbb{P}(\mathcal{E}_i^w \cap \mathcal{E}_j^w)=\mathbb{E}\left(\sum_{w \in W} \mathbbm{1}_{\mathcal{E}_i^w \cap \mathcal{E}_j^w}\right),
\end{equation}
where the number of vertices in $W_{i,j}$ is at least $\frac{c n^{5/3}}{100k_{i,j}}$ for all large \(n\).
Let us define the following random variable.
\begin{itemize}
    \item $N_{i,j}$ is number of $w \in W_{i,j}$ such that there exist $u_i,u_j \in \mathbb{Z}^2$ and $\gamma_i \in A_i$ and $\gamma_j \in A_j$ with $|\phi(u_i)+n| \leq \lceil \frac{n}{100k_{i,j}} \rceil$ and $|\phi(u_j)+n| \leq \lceil \frac{n}{100k_{i,j}} \rceil$ and $w \in \Gamma_{u_i}^{\gamma_i} \cap \Gamma_{u_j}^{\gamma_j}.$
\end{itemize}
We have
\begin{displaymath}
 \sum_{w \in W_{i,j}} \mathbbm{1}_{\mathcal{E}_i^w \cap \mathcal{E}_j^w} \leq N_{i,j}.
\end{displaymath}
We will show that there exists a constant $C>0$ (depending only on $\varepsilon$) such that for all $n$ sufficiently large
\begin{equation}
\label{expected_number_inside_parallelogram_bound}
    \mathbb{E}(N_{i,j}) \leq \frac{C n \log |k_{i,j}|}{k_{i,j}^3}.
\end{equation}
We define two equivalence classes. For $u_1, u_2 \in \mathbb{Z}^2$ and $\gamma^1_i, \gamma_i^2 \in A_i$ with $|\phi(u_1)+n| \leq \lceil \frac{n}{100k_{i,j}} \rceil$ and $|\phi(u_2)+n| \leq \lceil \frac{n}{100k_{i,j}} \rceil$ we say $(u_1,\gamma^1_i) \sim (u_2,\gamma^2_i)$ if $\Gamma_{u_1}^{\gamma^1_{i}}$ and $\Gamma_{u_2}^{\gamma^2_i}$ coincide inside $W_{i,j}.$ Let $\widetilde{E_i}$ denote the number of equivalence classes. Similarly, we define the equivalence classes corresponding to $A_j$ and denote it by $\widetilde{E_j}.$ Further for $u, v \in \mathbb{Z}^2$ with $|\phi(u)+n| \leq \lceil \frac{n}{100k_{i,j}} \rceil, |\phi(v)+n| \leq \lceil \frac{n}{100k_{i,j}} \rceil,\gamma_i$ (resp.\ $\gamma_j$) contained in $A_i$ (resp.\ $A_j$) let $I_{u,v, \gamma_,\gamma_j}$ denote the intersection size of $\Gamma_{u}^{\gamma_i}$ and $\Gamma_v^{\gamma_j}$ inside $W_{i,j}$.
Clearly,
\begin{displaymath}
    N_{i,j} \leq \widetilde{E_i}\widetilde{E_j} \max I_{u,v, \gamma_i, \gamma_j},
\end{displaymath}
where the maximum is taken over all $u,v \in \mathbb{Z}^2$ with $|\phi(u)+n| \leq \lceil \frac{n}{100k_{i,j}} \rceil$ and $|\phi(v)+n| \leq \lceil \frac{n}{100k_{i,j}} \rceil$ and $\gamma_i$ (resp.\ $\gamma_j$) lies in $A_i$ (resp.\ $A_j$). So, for $\ell \geq 1$
\[
\label{union_bound}
    \mathbb{P}\left(N_{i,j} \geq \frac{\ell n \log |k_{i,j}|}{k_{i,j}^3}\right)
    \leq \mathbb{P}(\widetilde{E_i} \geq \ell^{1/3})+ \mathbb{P}(\widetilde{E_j} \geq \ell^{1/3})+ \mathbb{P}\left(\max I_{u,v, \gamma_i,\gamma_j} \geq \frac{\ell^{1/3} n \log |k_{i,j}|}{k_{i,j}^3}\right).
\]
We start with an upper bound for the first term. Let us consider the following parallelogram (see Figures \ref{fig: second_sum_figure}, \ref{fig: second_sum_figure_1}).
\begin{itemize}
 \item  $\mathsf{R}_i^{\ell}$ is the parallelogram whose opposite sides lie on $\mathcal{L}_{-n-\lceil \frac{n}{100k_{i,j}}\rceil}$ (resp.\ $\mathcal{L}_{-n+\lceil \frac{n}{100k_{i,j}}\rceil}$), each of length $\ell^{1/96}n^{2/3}$ with midpoints being the intersection points of $y=\tan(\theta_i)x$ with $\mathcal{L}_{-n-\lceil \frac{n}{100k_{i,j}}\rceil}$ (resp.\ $\mathcal{L}_{-n+\lceil \frac{n}{100k_{i,j}}\rceil}$).
\end{itemize}
Similarly we define $\mathsf{R}_j^\ell$ by replacing $\theta_i$ with $\theta_j.$
\begin{figure}[!ht]
    \includegraphics[width=13 cm]{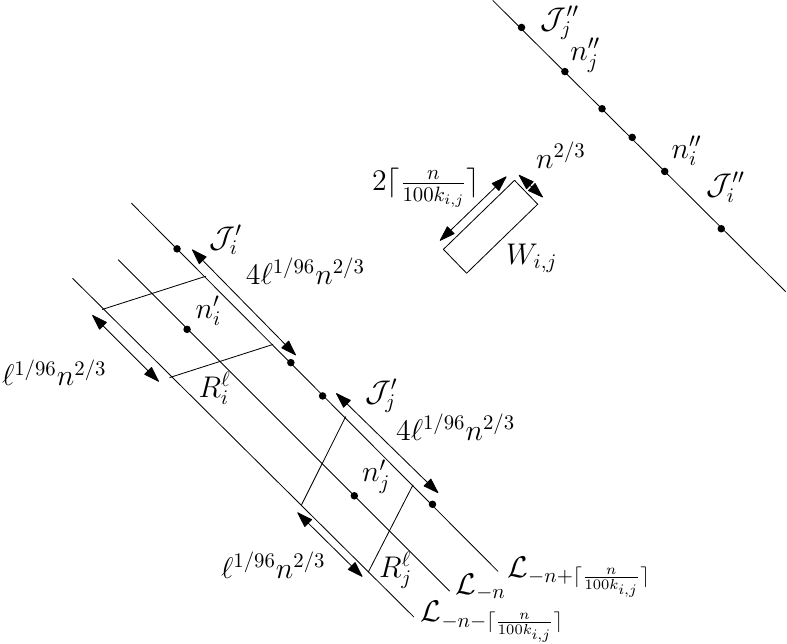}
    \caption{To find an upper bound for the second sum in \eqref{inclusion_exclusion} we fix $1 \leq i < j \leq [\frac{n^{1/3}}{2M}]$. Due to planarity and large transversal fluctuation there is a very small probability that there will be vertices $u$ (resp.\ $v$) outside $\mathsf{R}_i^{\ell}$ (resp.\ $\mathsf{R}_j^{\ell}$) with $|\phi(u)+n| \leq \lceil \frac{n}{100 k_{i,j}} \rceil$ (resp.\ $|\phi(v)+n| \leq \lceil \frac{n}{100 k_{i,j}} \rceil$) and $\gamma_i \in A_i$ (resp.\ $\gamma_j \in A_j$) such that $\Gamma_u^{\gamma_i}$ (resp.\ $\Gamma_v^{\gamma_j}$)  will intersect $W_{i,j}$. These are the events $\mathcal{F}_i$ and $\mathcal{F}_j$. Further, the same argument as in \cite[Lemma 3.1]{BBB23} shows that, there is a very small probability that for any geodesic starting from $\mathsf{R}_i^{\ell}$(resp.\ $\mathsf{R}_j^{\ell}$) and taking angle in $A_i$ (resp.\ $A_j$) will not intersect either $\mathcal{J}_i'$ or $\mathcal{J}_i''$ (resp.\ either $\mathcal{J}_j'$ or $\mathcal{J}_j''$). These are the events $\mathcal{G}_i$ and $\mathcal{G}_j.$ Further, due to Proposition \ref{coalescence_theorem}, we have with very small probability that there are more than $\ell^{1/3}$ distinct vertices in $W_{i,j}$ that have a geodesic starting from $\mathcal{J}_i'$ (resp.\ $\mathcal{J}_j'$) and ending at $\mathcal{J}_i''$ (resp.\ $\mathcal{J}_j''$). Finally, we invoke \cite[Lemma 3.2, Lemma 3.3]{BBB23} to conclude that with exponentially small probability any geodesic starting from $\mathcal{J}_i'$ (resp.\ $\mathcal{J}_j'$) and ending at $\mathcal{J}_i''$ (resp.\ $\mathcal{J}_j''$) will have more than $\frac{\ell n}{k_{i,j}^3}$ intersection size inside $W_{i,j}$. Combining all the above we get stretched exponential upper bound for each terms in \eqref{union_bound}.}
    \label{fig: second_sum_figure} 
\end{figure}

Let us consider the following two line segments $\mathcal{J}_i'$ (resp.\ $\mathcal{J}_i''$) with $4\ell^{1/96}n^{2/3}$ many vertices on $\mathcal{L}_{-n+\lceil \frac{n}{100k_{i,j}} \rceil}$ (resp.\ $\mathcal{L}_{n}$) with midpoints the intersection point of $y=\tan(\theta_i)x$ and $\mathcal{L}_{-n+\lceil \frac{n}{100k_{i,j}}\rceil}$ (resp.\ $n_i''$). Let us define following two events (see Figures \ref{fig: second_sum_figure}, \ref{fig: second_sum_figure_1}).
\begin{itemize}
 \item $\mathcal{F}_i:=\{ \exists u\notin\mathsf{R}_i^{\ell}\text{ and } \exists \gamma_i \in A_i: |\phi(u)+n| \leq \lceil \frac{n}{100 k_{i,j}} \rceil, \Gamma_u^{\gamma_i}$ intersects $W_{i,j}\}$.
    \item $\mathcal{G}_i:=\{\exists u \in \mathsf{R}_i^{\ell} \text{ and } \gamma_i \in A_i: \Gamma_u^{\gamma_i}$ does not intersect either $\mathcal{J}_i'$ or $\mathcal{J}_i''\}$.
\end{itemize}
Similarly, we define $\mathcal{F}_j, \mathcal{G}_j$ with $\theta_j$ in place of $\theta_i.$ We have the following lemma.
\begin{lemma}
\label{lemma: corrected lemma}There exist $C,c>0$ (depending only on $\varepsilon$) such that for sufficiently large $\ell$ and $n$
\begin{displaymath}
    \mathbb{P}(\mathcal{F}_i \cup \mathcal{F}_j) \leq Ce^{-c \ell^{3/96}}.
\end{displaymath}
\end{lemma}
\begin{figure}[!ht]
    \includegraphics[width=13 cm]{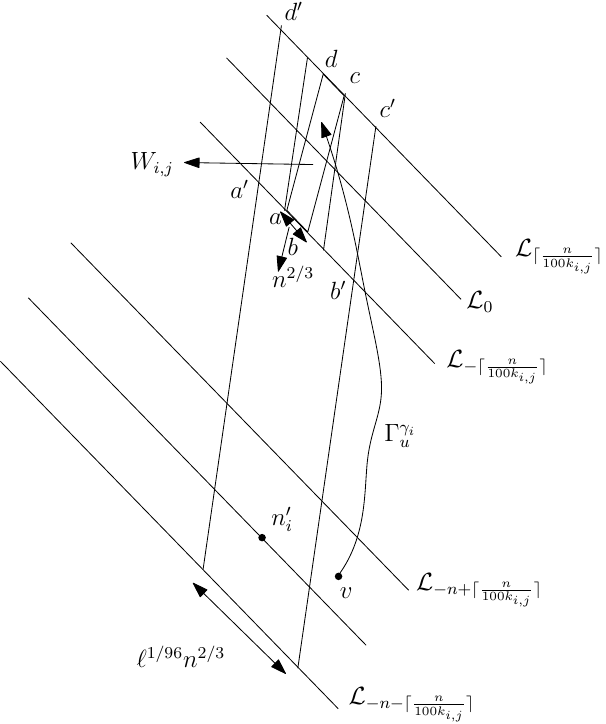}
    \caption{The figure shows how to obtain an upper bound for $\mathbb{P}(\mathcal{F}_i)$ when $\theta_i>\theta_j> \frac \pi4.$ Although $W_{i,j}$ is aligned in the direction $\theta_j$, if for some $u \notin \overline{\mathsf{R}_i^\ell}$ and $\gamma_i \in A_i$ intersect $W_{i,j}$ then it will have large transversal fluctuation of order at least $\sim (\ell^{1/96}-1)n^{2/3}$ as shown in the figure. Thus, $\mathbb{P}(\mathcal{F}_i)$ will be exponentially small in $\ell$. The other cases of $\theta_i,\theta_j$ follow similarly.}
    \label{fig: second_sum_figure_1} 
\end{figure}

\begin{proof}
 We assume the case when $\theta_i>\theta_j>\frac{\pi}{4}$. The other cases will follow similarly. The upper bound for $\mathbb{P}(\mathcal{F}_j)$ is immediate from Proposition \ref{transversal_fluctuation_of_semi_infinite_geodesic} as the parallelogram $W_{i,j}$ is aligned in the direction $\theta_j.$ For the event $\mathcal{F}_i,$ first define $\overline{\mathsf{R}_i^{\ell}}$ as the parallelogram whose opposite sides lie on $\mathcal{L}_{-n-\lceil \frac{n}{100k_{i,j}}\rceil}$ (resp.\ $\mathcal{L}_{\lceil \frac{n}{100k_{i,j}}\rceil}$), each of length $\ell^{1/96}n^{2/3}$ with midpoints at the intersections of $y=\tan(\theta_i)x$ and $\mathcal{L}_{-n-\lceil \frac{n}{100k_{i,j}}\rceil}$ (resp.\ $\mathcal{L}_{\lceil \frac{n}{100k_{i,j}}\rceil}$, see Figure \ref{fig: second_sum_figure_1}). Although $W_{i,j}$ is aligned in the direction $\theta_j$ which is different from $\theta_i$, we now show that it is still properly contained in the parallelogram $\overline{\mathsf{R}_i^{\ell}}$ for sufficiently large \(\ell\).

  To see this, it is enough to verify the containment on the lines \(\mathcal L_{\lceil\frac n{100k_{i,j}}\rceil}\) and \(\mathcal L_{-\lceil\frac n{100k_{i,j}}\rceil}\). As these lines have direction \(\frac34\pi\), containment of \(W_{i,j}\) in \(\overline{\mathsf{R}_i^{\ell}}\) is implied by the corresponding ordering of the horizontal coordinates for the intersections of the respective parallelograms with these two lines. We show exactly this below. By definition, the sides of \(W_{i,j}\) on these lines have midpoints that come from the system
  \[
   \left.
    \begin{aligned}
     y&=\tan(\theta_j)x\\
     x+y&=\lceil\frac n{100k_{i,j}}\rceil
    \end{aligned}
   \right\}
   \qquad\text{and}\qquad
   \left.
    \begin{aligned}
     y&=\tan(\theta_j)x\\
     x+y&=-\lceil\frac n{100k_{i,j}}\rceil
    \end{aligned}
   \right\}.
  \]
  That is, the upper side of \(W_{i,j}\), of side length \(n^{2/3}\), spans through horizontal coordinates
  \[
   x\in\Bigl(\frac{\lceil\frac n{100k_{i,j}}\rceil}{1+\tan\theta_j}-\frac1{2\sqrt2}n^{2/3},\,\frac{\lceil\frac n{100k_{i,j}}\rceil}{1+\tan\theta_j}+\frac1{2\sqrt2}n^{2/3}\Bigr)
  \]
  on \(\mathcal L_{\lceil\frac n{100k_{i,j}}\rceil}\), while the lower side spans
  \[
   x\in\Bigl(\frac{-\lceil\frac n{100k_{i,j}}\rceil}{1+\tan\theta_j}-\frac1{2\sqrt2}n^{2/3},\,\frac{-\lceil\frac n{100k_{i,j}}\rceil}{1+\tan\theta_j}+\frac1{2\sqrt2}n^{2/3}\Bigr)
  \]
  on \(\mathcal L_{-\lceil\frac n{100k_{i,j}}\rceil}\). To simplify notation we increase this for all large \(n\) by dropping the constant factors
  on \(n^{2/3}\).

  The parallelogram \(\overline{\mathsf{R}_i^{\ell}}\) has side length at least \(2\sqrt2\ell^{1/97}n^{2/3}\) for large \(\ell\) (i.e., by slightly decreasing the power on \(\ell\) we don't need to carry constant factors below). We then conclude, similarly to how we did for \(W_{i,j}\), that the intersection of \(\overline{\mathsf{R}_i^{\ell}}\) with the line \(\mathcal L_{\lceil\frac n{100k_{i,j}}\rceil}\) contains the horizontal coordinates
  \[
   x\in\Bigl(\frac{\lceil\frac n{100k_{i,j}}\rceil}{1+\tan\theta_i}-\ell^{1/97}n^{2/3},\,\frac{\lceil\frac n{100k_{i,j}}\rceil}{1+\tan\theta_i}+\ell^{1/97}n^{2/3}\Bigr)
  \]
  and, with the line \(\mathcal L_{-\lceil\frac n{100k_{i,j}}\rceil}\), the horizontal coordinates
  \[
   x\in\Bigl(\frac{-\lceil\frac n{100k_{i,j}}\rceil}{1+\tan\theta_i}-\ell^{1/97}n^{2/3},\,\frac{-\lceil\frac n{100k_{i,j}}\rceil}{1+\tan\theta_i}+\ell^{1/97}n^{2/3}\Bigr).
  \]

  In view of the above, to show containment it is enough to demonstrate that
  \[
   \Bigl(\frac{\lceil\frac n{100k_{i,j}}\rceil}{1+\tan\theta_j}-n^{2/3},\,\frac{\lceil\frac n{100k_{i,j}}\rceil}{1+\tan\theta_j}+n^{2/3}\Bigr)
   \subseteq\Bigl(\frac{\lceil\frac n{100k_{i,j}}\rceil}{1+\tan\theta_i}-\ell^{1/97}n^{2/3},\,\frac{\lceil\frac n{100k_{i,j}}\rceil}{1+\tan\theta_i}+\ell^{1/97}n^{2/3}\Bigr)
  \]
  and
  \[
   \Bigl(\frac{-\lceil\frac n{100k_{i,j}}\rceil}{1+\tan\theta_j}-n^{2/3},\,\frac{-\lceil\frac n{100k_{i,j}}\rceil}{1+\tan\theta_j}+n^{2/3}\Bigr)
   \subseteq\Bigl(\frac{-\lceil\frac n{100k_{i,j}}\rceil}{1+\tan\theta_i}-\ell^{1/97}n^{2/3},\,\frac{-\lceil\frac n{100k_{i,j}}\rceil}{1+\tan\theta_i}+\ell^{1/97}n^{2/3}\Bigr).
  \]
  Due to symmetry, these four inequalities reduce to the below two:
  \begin{align}
   \frac{\lceil\frac n{100k_{i,j}}\rceil}{1+\tan\theta_j}-n^{2/3}&\ge\frac{\lceil\frac n{100k_{i,j}}\rceil}{1+\tan\theta_i}-\ell^{1/97}n^{2/3},\label{eq:parleft}\\
   \frac{\lceil\frac n{100k_{i,j}}\rceil}{1+\tan\theta_j}+n^{2/3}&\le\frac{\lceil\frac n{100k_{i,j}}\rceil}{1+\tan\theta_i}+\ell^{1/97}n^{2/3}.\label{eq:parright}
  \end{align}

  \eqref{eq:parleft} is easy in the case \(\theta_i>\theta_j>\frac\pi4\) since for large \(\ell\) the terms on the two sides of the inequality dominate each other pairwise.

  \eqref{eq:parright} is more involved. Rearranging, we need to show that
  \[
   \begin{aligned}
    \lceil\frac n{100k_{i,j}}\rceil\Bigl(\frac1{1+\tan\theta_j}-\frac1{1+\tan\theta_i}\Bigr)&\le\bigl(\ell^{1/97}-1\bigr)n^{2/3},\\
    \lceil\frac n{100k_{i,j}}\rceil\cdot\frac{\tan\theta_i-\tan\theta_j}{(1+\tan\theta_j)(1+\tan\theta_i)}&\le\bigl(\ell^{1/97}-1\bigr)n^{2/3}.
   \end{aligned}
  \]
  Dropping the second denominator can only increase the left-hand side. Recalling \(1\le j<i\le\bigl[\frac{n^{1/3}}{2M}\bigr]\) and \(k_{i,j}=M(i-j)\), we also see that the first fraction is at least in the order of \(n^{2/3}\). Hence we can also get rid of the integer part by adding in a factor of 2. It is therefore enough to show
  \[
   \frac n{50k_{i,j}}\cdot(\tan\theta_i-\tan\theta_j)\le\bigl(\ell^{1/97}-1\bigr)n^{2/3}.
  \]
  Recall that the angles are each in the range \((\frac\pi4,\,\frac\pi2-\varepsilon)\). The derivative of \(\tan\) is bounded here, uniformly in all parameters except \(\varepsilon\). Thus, with \(k_{i,j}=M(i-j)\) and the definition \(\theta_i=\frac{Mi(\frac\pi2-2\varepsilon)}{n^{1/3}}\), we now need to prove
  \[
   \text{const}\cdot\frac n{i-j}\cdot\frac{i-j}{n^{1/3}}\le\bigl(\ell^{1/97}-1\bigr)n^{2/3},
  \]
  which is trivially true for large \(\ell\).
 
 Now we show that if there is a geodesic starting outside of $R_i^\ell$ and taking angle in $A_i$ intersects $W_{i,j}$ then it will imply large transversal fluctuation. Note that from the above discussion, the end points of $W_{i,j}$ are (see Figure \ref{fig: second_sum_figure_1})
 \begin{align*}
  a:&=\left(\frac{-\lceil\frac{n}{100k_{i,j}} \rceil}{1+\tan(\theta_j)}-\frac{1}{2\sqrt{2}}n^{2/3},\frac{-\lceil\frac{n}{100k_{i,j}} \rceil \tan(\theta_j)}{1+\tan(\theta_j)}+\frac{1}{2 \sqrt{2}}n^{2/3}\right),\\
  b:&=\left(\frac{-\lceil\frac{n}{100k_{i,j}} \rceil}{1+\tan(\theta_j)}+\frac{1}{2 \sqrt{2}}n^{2/3},\frac{-\lceil\frac{n}{100k_{i,j}} \rceil \tan(\theta_j)}{1+\tan(\theta_j)}-\frac{1}{2\sqrt{2}}n^{2/3}\right),\\
  c:&= \left(\frac{\lceil\frac{n}{100k_{i,j}} \rceil}{1+\tan(\theta_j)}-\frac{1}{2 \sqrt{2}}n^{2/3},\frac{\lceil\frac{n}{100k_{i,j}} \rceil \tan(\theta_j)}{1+\tan(\theta_j)}+\frac{1}{2\sqrt{2}}n^{2/3}\right),\\
  d:&=\left(\frac{\lceil\frac{n}{100k_{i,j}} \rceil}{1+\tan(\theta_j)}+\frac{1}{2\sqrt{2}}n^{2/3},\frac{\lceil\frac{n}{100k_{i,j}} \rceil \tan(\theta_j)}{1+\tan(\theta_j)}-\frac{1}{2\sqrt{2}}n^{2/3}\right).
 \end{align*}
 Similarly, the corresponding vertices of $R_i^\ell$ restricted between $\mathcal{L}_{-\lceil\frac{n}{100k_{i,j}} \rceil}$ and $\mathcal{L}_{\lceil\frac{n}{100k_{i,j}} \rceil}$  are (see Figure \ref{fig: second_sum_figure_1})
 \begin{align*}
     a':&=\left(\frac{-\lceil\frac{n}{100k_{i,j}} \rceil}{1+\tan(\theta_j)}-\frac{\ell^{1/96}n^{2/3}}{2 \sqrt{2}},\frac{-\lceil\frac{n}{100k_{i,j}} \rceil \tan(\theta_j)}{1+\tan(\theta_j)}+\frac{\ell^{1/96}n^{2/3}}{2\sqrt{2}} \right),\\
     b':&=\left(\frac{-\lceil\frac{n}{100k_{i,j}} \rceil}{1+\tan(\theta_j)}+\frac{\ell^{1/96}n^{2/3}}{2\sqrt{2}},\frac{-\lceil\frac{n}{100k_{i,j}} \rceil \tan(\theta_j)}{1+\tan(\theta_j)}-\frac{\ell^{1/96}n^{2/3}}{2\sqrt{2}} \right),\\
     c':&=\left(\frac{\lceil\frac{n}{100k_{i,j}} \rceil}{1+\tan(\theta_j)}-\frac{\ell^{1/96}n^{2/3}}{2\sqrt{2}},\frac{\lceil\frac{n}{100k_{i,j}} \rceil \tan(\theta_j)}{1+\tan(\theta_j)}+\frac{\ell^{1/96}n^{2/3}}{2\sqrt{2}} \right),\\
     d':&=\left(\frac{\lceil\frac{n}{100k_{i,j}} \rceil}{1+\tan(\theta_j)}+\frac{\ell^{1/96}n^{2/3}}{2\sqrt{2}},\frac{\lceil\frac{n}{100k_{i,j}} \rceil \tan(\theta_j)}{1+\tan(\theta_j)}-\frac{\ell^{1/96}n^{2/3}}{2\sqrt{2}} \right).
 \end{align*}
 Then  
 \[
 \psi(a)-\psi(a')=-\frac{2\lceil \frac{n}{100 k_{i,j}} \rceil \left(\tan(\theta_i)-\tan(\theta_j) \right)}{(1+\tan(\theta_i))(1+\tan(\theta_j))}+\frac{1}{\sqrt{2}} \left(\ell^{1/96} n^{2/3}-n^{2/3}\right).
 \]
 from the above discussion, this difference is much larger than $(\ell^{1/97}-1)n^{2/3}$. Similarly, we compute $\psi(b')-\psi(b),\psi(c)-\psi(c')$ and $\psi(d')-\psi(d)$ to conclude the following:

 If $\exists u\notin\mathsf{R}_i^{\ell} \text{ and } \exists \gamma_i \in A_i: |\phi(u)+n| \leq \lceil \frac{n}{100 k_{i,j}} \rceil, \Gamma_u^{\gamma_i}$ intersects $W_{i,j}$ then, as $A_i$ is of length $\frac{1}{n^{1/3}}$, $\Gamma_u^{\gamma_i}$ will have transversal fluctuation at least $\sim (\ell^{1/97}-1)n^{2/3}$ from the time level \(-n+\lceil \frac{n}{100 k_{i,j}} \rceil\) to somewhere in between $-\lceil \frac{n}{100 k_{i,j}} \rceil$ and $\lceil \frac{n}{100 k_{i,j}} \rceil$ (this is precisely the reason for the choice of the direction of $W_{i,j}$ (see Figure \ref{fig: second_sum_figure_1}). By an easy planarity argument now the probability of this event can be bounded by the desired upper bound as in the statement of the lemma.
\end{proof}
Further, we have the following lemma.
\begin{lemma} 
\label{similar_first_lemma}For sufficiently large $n,\ell$ there exist constants $C,c>0$ (depending only on $\varepsilon$) such that
\begin{displaymath}
    \mathbb{P}(\mathcal{G}_i \cup \mathcal{G}_j) \leq Ce^{-c \ell ^{3/96}}.
\end{displaymath}
    
\end{lemma}
\begin{proof}
 This lemma is already proved in \cite[Lemma 3.1]{BBB23}, the reader can refer to this proof. We briefly describe the idea here. We consider two points $a$ and $b$ with $\psi(a)<\psi(b)$ on $\mathcal{L}_{-n-\lceil \frac{n}{100k_{i,j}}\rceil}$ which are $2\ell^{1/96}n^{2/3}$ distance away (in the space direction) from the end points of $R_i^\ell$. Now, by transversal fluctuation estimate we can say that $\Gamma_a^{\beta_i}$ and $\Gamma_b^{\alpha_i}$ always stay $\ell^{1/96}n^{2/3}$ distance away (in the space direction) from the rectangle $R_i^\ell$ with probability at least $1-e^{-c\ell^{3/96}}$. On this event, every geodesics starting in $R_i^\ell$ will be sandwiched between $\Gamma_a^{\beta_i}$ and $\Gamma_b^{\alpha_i}$. So, on this event by planarity argument the event $\mathcal{G}_i$ will imply large transversal fluctuation of either $\Gamma_a^{\beta_i}$ or $\Gamma_b^{\alpha_i}$. This gives us the desired upper bound.
\end{proof}
So, we have 
\begin{displaymath}
    \mathbb{P}(\widetilde{E_i} \geq \ell^{1/3}) \leq Ce^{-c \ell ^{3/96}}+ \mathbb{P}(\{\widetilde{E_i} \geq \ell^{1/3}\} \cap (\mathcal{F}_i \cup \mathcal{G}_i)^c).
\end{displaymath}
Note that on the second event in the right hand side there are more than $\ell^{1/3}$ distinct equivalence classes of geodesics starting from $\mathcal{J}_i'$ and ending at $\mathcal{J}_i''$.
Using Proposition \ref{coalescence_theorem}, for $\ell < n^{0.03}$ sufficiently large (depending on $\varepsilon$) we have
\begin{displaymath}
    \mathbb{P}(\{\widetilde{E_i} \geq \ell^{1/3}\} \cap (\mathcal{F}_i \cup \mathcal{G}_i)^c) \leq Ce^{-c \ell^{1/384}}.
\end{displaymath}
Same argument shows for $\ell <n^{0.03}$ sufficiently large (depending on $\varepsilon$) we have 
\begin{displaymath}
    \mathbb{P}(\{\widetilde{E}_j \geq \ell^{1/3}\}) \leq Ce^{-c \ell^{1/384}}.
\end{displaymath}
For the last term in \eqref{union_bound} we consider the following.
\begin{align*}
    \mathbb{P}\left(\max I_{u,v, \gamma_i, \gamma_j} \geq \frac{\ell^{1/3}n \log |k_{i,j}|}{k_{i,j}^3}\right) \leq \mathbb{P}(\mathcal{F}_i \cup \mathcal{G}_i \cup \mathcal{F}_j \cup \mathcal{G}_j)+ \\ \mathbb{P}\left(\left \{\max I_{u,v, \gamma_i,\gamma_j} \geq \frac{\ell^{1/3} n \log |k_{i,j}|}{k_{i,j}^3}\right\} \cap (\mathcal{F}_i \cup \mathcal{G}_i \cup \mathcal{F}_j \cup \mathcal{G}_j)^c\right).
\end{align*}
Clearly, we have 
\begin{displaymath}
    \mathbb{P}(\mathcal{F}_i \cup \mathcal{G}_i \cup \mathcal{F}_j \cup \mathcal{G}_j) \leq 2Ce^{-c \ell^{3/96}}.
\end{displaymath}
On the event $\left \{\max I_{u,v, \gamma_i,\gamma_j} \geq \frac{\ell^{1/3} n \log |k_{i,j}|}{k_{i,j}^3}\right \} \cap (\mathcal{F}_i \cup \mathcal{G}_i \cup \mathcal{F}_j \cup \mathcal{G}_j)^c$, we want to estimate the maximum size of the intersection of geodesics inside $W_{i,j}$ that are starting from $\mathcal{J}_i'$ (resp.\ $\mathcal{J}_j'$) and ending at $\mathcal{J}_i''$ (resp.\ $\mathcal{J}_j''$). For this we use Lemma \ref{lemma: 3.2, 3.3}.
    Hence, we have that there exist constants $C,c$ (depending on $\varepsilon$) such that for large enough $\ell$ and $n$
    \begin{displaymath}
        \mathbb{P}\left(\left \{\max I_{u,v, \gamma_i,\gamma_j} \geq \frac{\ell^{1/3} n \log |k_{i,j}|}{k_{i,j}^3}\right\} \cap (\mathcal{F}_i \cup \mathcal{G}_i \cup \mathcal{F}_j \cup \mathcal{G}_j)^c \right) \leq Ce^{-c \ell^{1/3}}.
    \end{displaymath}
Hence, combining the above we get that there exists a constant $C>0$ (depending on $\varepsilon$) such that for all large enough $n$ we have
\begin{displaymath}
\mathbb{E}(N_{i,j}) \leq \frac{C n \log |k_{i,j}|}{k_{i,j}^3}.
\end{displaymath}
hence, from \eqref{averaging} we have that there exists $C_2>0$ such that for sufficiently large $n$
\begin{equation}
\label{second_sum_upper_bound}
    \mathbb{P}(\mathcal{E}_i \cap \mathcal{E}_j) \leq \frac{C_2 \log |k_{i,j}|}{k_{i,j}^2n^{2/3}}.
\end{equation}
We now go back to \eqref{inclusion_exclusion_for_depth}. From \eqref{first_sum_lower_bound_depth} and \eqref{second_sum_upper_bound} we have that there exists constants $c_1,C_2>0$, large $M$ such that for sufficiently large $n$(all depending only on $\varepsilon$)
\begin{displaymath}
    \mathbb{P}(D \geq n) \geq \frac{c_1}{M^2n^{1/3}}-\sum_{1 \leq i < j \leq [\frac{n^{1/3}}{2M}]} \frac{C_2 \log (M|i-j|)}{M^2|i-j|^2n^{2/3}}.
\end{displaymath}
Hence, we have for some $c>0,C>0$
\begin{displaymath}
    \mathbb{P}(D \geq n) \geq \frac{c}{M^2n^{1/3}}-\frac{C}{M^3n^{1/3}}.
\end{displaymath}
As $C,c$ depend only on $\varepsilon$ (as we have chosen all the interval length to be $\frac{1}{n^{1/3}}$) we can choose $M$ large enough so that there exists $c>0$ such that for sufficiently large $n$
\begin{displaymath}
    \mathbb{P}(D \geq n) \geq \frac{c}{n^{1/3}}.
\end{displaymath}
This completes the proof of lower bound in Theorem \ref{depth_bounds_theorem}. \qed\\

\section{Bounds for number of cars}
\label{Bounds for Number of Cars}
In this section we prove the bounds for the tail distribution of $N$. Recall that $N$ was defined to be the number of cars that pass through the origin, and \(D\) was the greatest distance a car travelled to the origin. More precisely,
\[
 \begin{aligned}
  N:&=\# \left \{u \in \mathbb{Z}^2: \boldsymbol{0} \in \Gamma_u^{\theta_u} \right\},\\
  D:&=\max\{n\in\mathbb{N}:\exists u\in\mathcal L_{-n}\text{ such that }\boldsymbol0\in\Gamma_u^{\theta_u}\}.
 \end{aligned}
\]
In particular, we will prove there exist $c_1, C_1>0$ (depending on $\varepsilon$) such that for all $n$ sufficiently large
\[
\frac{c_1}{n^{1/3}} \leq \mathbb{P} \left( N \geq n^{4/3}\right) \leq \frac{C_1 \log n}{n^{1/3}}.
\]
Note that, this will imply Theorem \ref{number of cars upper bound}.
\subsection{Proof of Theorem \ref{number of cars upper bound} upper bound}
First we restrict ourselves to smaller events. We have
\begin{displaymath}
    \mathbb{P}(N \geq n^{4/3}) \leq \mathbb{P}(D \geq 2n)+\mathbb{P}\left(\{N \geq n^{4/3}\} \cap \{D \leq 2n\}\right).
\end{displaymath}
We have already proved 
\begin{displaymath}
    \mathbb{P}(D \geq 2n) \leq \frac{C}{n^{1/3}}.
\end{displaymath}
So, we will only consider $\mathbb{P}(\{N \geq n^{4/3}\} \cap \{D \leq 2n\})$.
We want to apply an averaging argument. Let $V$ denote the line segment with $\lceil n^{2/3} \rceil $ vertices on $\mathcal{L}_0$ with midpoint $\boldsymbol{0}$. For $v \in V$ define the following random variables. 
\begin{displaymath}
    D_v:=\max\{n \in \mathbb{N}: \exists u \in \mathcal{L}_{-n} \text{ such that } v \in \Gamma_u^{\theta_u} \}
\end{displaymath}
\begin{displaymath}
 N_v:=\#\left \{u \in \mathbb{Z}^2: v \in \Gamma_u^{\theta_u}\right\}.
\end{displaymath}
Then we have for all $v \in V$
\begin{displaymath}
    \mathbb{P}(N \geq n^{4/3}, D \leq 2n)=\mathbb{P}(N_v \geq n^{4/3}, D_v \leq 2n).
\end{displaymath}
Hence,
\begin{displaymath}
    \lceil n^{2/3} \rceil \mathbb{P}(N \geq n^{4/3}, D \leq 2n) = \sum_{v \in V}\mathbb{P}(N_v \geq n^{4/3}, D_v \leq 2n)=\mathbb{E}(\widehat{N}),
\end{displaymath}
where $\widehat{N}$ is defined as follows. 
\begin{displaymath}
    \widehat{N}:=\sum_{v \in V}\mathbbm{1}_{\{N_v \geq n^{4/3}, D_v \leq 2n\}}.
\end{displaymath}
For $\ell \geq 1$ we want to estimate $\mathbb{P}(\widehat{N}\geq \ell n^{1/3} \log n)$. We define some more random variables. 
\begin{displaymath}
 \widetilde{N}:=\#\{u \in \mathbb{Z}^2: -2n \leq \phi(u) \leq 0 \text{ and }\exists v \in V \text{ such that } v \in \Gamma_u^{\theta_u}\}.
\end{displaymath}
Observe that 
\begin{displaymath}
    \{\widehat{N}\geq \ell n^{1/3 }\log n\} \subset \{ \widetilde{N} \geq \ell n^{5/3} \log n\}.
\end{displaymath}
Same as before, for $1 \leq i \leq \lfloor n^{1/3} \rfloor$ and for $c_n=\frac{n^{1/3}}{\lfloor n^{1/3}\rfloor }$ we partition the interval $(\varepsilon, \frac{\pi}{2}-\varepsilon$) into intervals $A_i$ each of equal length $\frac{c_n(\frac{\pi}{2}-2 \varepsilon)}{n^{1/3}}$. For $1 \leq i \leq \lfloor n^{1/3} \rfloor$ we define the following random variables.
\begin{displaymath}
 \widetilde{N}_i:=\# \left \{u \in \mathbb{Z}^2: -2n \leq \phi(u) \leq 0, \theta_u \in A_i \text{ and } \exists v \in V \text{ such that } v \in \Gamma_u^{\theta_u} \right\}.
\end{displaymath}
Clearly, we have
\begin{displaymath}
    \left \{\widetilde{N} \geq \ell n^{5/3} \log n \right \} \subset \bigcup_{i=1}^{\lfloor n^{1/3} \rfloor}\left \{\widetilde{N}_i \geq \ell n^{4/3} \log n \right \}.
\end{displaymath}
We will show that for each $1 \leq i \leq \lfloor n^{1/3} \rfloor$ there exist $C,c>0$ (depending only on $\varepsilon$) such that for sufficiently large $n$ (depending only on $\varepsilon$) and $\ell<n^{1/6}$
\begin{equation}
\label{exponential bound}
    \mathbb{P}\left (\widetilde{N}_i \geq \ell n^{4/3} \right ) \leq Ce^{-c \ell}.
\end{equation}
We fix $1 \leq i \leq \lfloor n^{1/3} \rfloor$. Let $ \alpha_i, \beta_i, \theta_i$ denote the left end point, right end point and midpoint of $A_i$ respectively and $n_i \in \mathcal{L}_{-2n}$ be the intersection point of $\mathcal{L}_{-2n}$ and the line $y=(\tan \theta_i) x$. Consider the line segment $\widetilde{V}$ with $\lfloor \frac{\ell}{32} n^{2/3} \rfloor$ vertices on $\mathcal{L}_{-2n}$ with midpoint $n_i$. Let $v_1,v_2$ be the end points of $\widetilde{V}$ with $\psi(v_1) \leq \psi(v_2).$ Consider $\Gamma_{v_1}^{\alpha_i}$ and $\Gamma_{v_2}^{\beta_i}$. Let $\mathsf{R}_i$ denote the parallelogram whose one pair of opposite sides are line segments with $\lfloor \frac{\ell}{16} n^{2/3} \rfloor$ many vertices and they lie on $\mathcal{L}_{-2n}$ (resp.\ $\mathcal{L}_{0}$) with midpoint $n_i$ (resp.\ $\boldsymbol{0}$) (see Figure \ref{fig: upper_bound_figure}).
\begin{figure}[!ht]
    \includegraphics[width=13 cm]{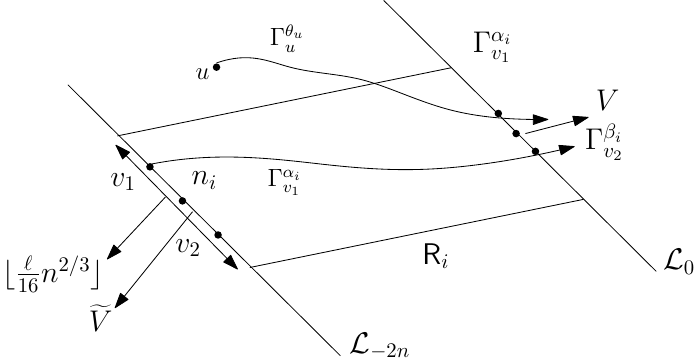}
    \caption{To prove the upper bound in Theorem \ref{number of cars upper bound} we fix $1 \leq i \leq \lfloor n^{1/3} \rfloor$. The event that there are more than $\ell n^{4/3} $ many $u$ with $\theta_u \in A_i$ and $\Gamma_u^{\theta_u}$ intersecting $V$ can happen in two ways. Either the above $u$ lies outside $\mathsf{R}_i$. For sufficiently large $\ell$ this will imply large transversal fluctuation for  either $\Gamma_{v_1}^{\alpha_i}$ or $\Gamma_{v_2}^{\beta_i}$ (the case when $\Gamma^{\alpha_i}_{v_1}$ will have large transversal fluctuation is shown in the figure). Proposition \ref{transversal_fluctuation_of_semi_infinite_geodesic} gives a stretched exponential upper bound for this event. The other possibility is that there are more than $\ell n^{4/3}$ many vertices inside $\mathsf{R}_i$ that choose an angle in $A_i$. Using Hoeffding's inequality we get a stretched exponential upper bound for this event. Combining the above we get \eqref{exponential bound}.}
    \label{fig: upper_bound_figure} 
\end{figure}
Let $\mathcal{H}_i$ be the event defined as follows.
\begin{itemize}
 \item $\mathcal{H}_i:= \left \{ \exists u \in \mathbb{Z}^2 \setminus \mathsf{R}_i \text{ such that } \theta_u \in A_i \text{ and } v \in \Gamma^{\theta_u}_u \text{ for some } v \in V \right \}$.
\end{itemize}
Then 
\begin{equation}
\label{breaking_equation}
     \mathbb{P} \left (\widetilde{N}_i \geq \ell n^{4/3} \right ) \leq \mathbb{P}(\mathcal{H}_i)+\mathbb{P}\left (\left \{\widetilde{N}_i \geq \ell n^{4/3}\right \} \cap (\mathcal{H}_i)^c \right).
\end{equation}
Note that by planarity the event $\mathcal{H}_i$ can happen in two ways. Either $\Gamma_{v_1}^{\alpha_i}$ or $\Gamma_{v_2}^{\beta_i}$ will leave $\mathsf{R}_i$ or they will have large transversal fluctuation from the straight lines starting from $v_1$ in the direction $\alpha_i$ and starting from $v_2$ in the direction $\beta_i$ respectively. (see Figure \ref{fig: upper_bound_figure}). Since, $(\beta_i-\alpha_i)$ is of order $\frac{1}{n^{1/3}}$, for sufficiently large $\ell$ (depending on $\varepsilon$), the event $\mathcal{H}_i$ will imply that either $\Gamma_{v_1}^{\alpha_i}$ or $\Gamma_{v_2}^{\beta_i}$ will have transversal fluctuation larger than $\frac{\ell}{64}n^{2/3}$. By Proposition \ref{transversal_fluctuation_of_semi_infinite_geodesic} below we have 
Hence, for sufficiently large $\ell$
\begin{displaymath}
    \mathbb{P}(\mathcal{H}_i) \leq Ce^{-c \ell^3},
\end{displaymath}
for some $C,c>0$ depending on $\varepsilon.$
We consider the second term in \eqref{breaking_equation}. On the event $\left \{\widetilde{N}_i \geq \ell n^{4/3}\right \} \cap (\mathcal{H}_i)^c$ there are at least $\ell n^{4/3}$ many vertices inside $\mathsf{R}_i$ that chooses an angle from $A_i$. For large $\ell$ and $n$ there are at most $\frac{\ell}{4} n^{5/3}$ many vertices inside $\mathsf{R}_i$. Consider the collection of i.i.d.\ random variables $\{\mathbbm{1}_{\{\theta_u \in A_i\}}\}_{u \in \mathsf{R}_i}$ with mean $\frac{c_n}{n^{1/3}}$. So, 
\begin{displaymath}
    \{\widetilde{N}_i \geq \ell n^{4/3}\} \cap (\mathcal{H}_i)^c \subset \left \{\sum_{u \in \mathsf{R}_i}\mathbbm{1}_{\{\theta_u \in A_i\}} \geq \ell n^{4/3}\right \}.
\end{displaymath}
By Hoeffding's inequality we have, for some $c>0$ 
\[
   \mathbb{P}\left(\sum_{u \in \mathsf{R}_i}\mathbbm{1}_{\{\theta_u \in A_i\}} \geq \ell n^{4/3}\right)\leq  \mathbb{P}\left(\sum_{u \in \mathsf{R}_i}\mathbbm{1}_{\{\theta_u \in A_i\}}-\mathbb{E}\left (\sum_{u \in \mathsf{R}_i}\mathbbm{1}_{\{\theta_u \in A_i\}}\right ) \geq \frac{\ell}{2}n^{4/3}\right) \leq e^{-c \ell}.
\]
Hence, we proved \eqref{exponential bound}.
Now, for $\ell<n^{1/6}$ sufficiently large
\begin{displaymath}
    \mathbb{P}\left (\widetilde{N}_i \geq \ell n^{4/3} \log n \right ) \leq Ce^{-c \ell \log n}.
\end{displaymath}
So, we have
\[
    \mathbb{P}\left(\bigcup_{i=1}^{\lfloor n^{1/3}\rfloor} \{\widetilde{N}_i \geq \ell n^{4/3} \log n \}\right)
    \leq \sum_{i=1}^{\lfloor n^{1/3} \rfloor}\mathbb{P}(\widetilde{N}_i \geq \ell n^{4/3} \log n) \leq Cn^{1/3}e^{-c \ell \log n} \leq C'e^{-c' \ell}.
\]
Thus,
\begin{displaymath}
    \mathbb{P}\left (\widehat{N} \geq \ell n^{1/3} \log n \right) \leq C'e^{-c' \ell}.
\end{displaymath}
So, we have that there exists $\ell_0$ (depending only on $\varepsilon$) and $\widetilde{C}$ (depending only on $\varepsilon$)
\[
    \mathbb{E}\left(\frac{\widehat{N}}{n^{1/3}\log n}\right) \leq \sum_{\ell \geq 1}\mathbb{P}(\widehat{N} \geq \ell n^{1/3} \log n)
    \leq  \ell_0 + \sum_{\ell_0 \leq \ell < n^{1/6}} C' e^{-c' \ell}+ \sum_{n^{1/6} < \ell < \frac{n^{1/3}}{\log n}}C'e^{-c'n^{1/6}}< \widetilde{C}.
\]
This proves the upper bound in Theorem \ref{number of cars upper bound} after an \(n\leftrightarrow n^{4/3}\) change of variable. \qed\\
\subsection{Proof of Theorem \ref{number of cars upper bound} lower bound} Same as before, first we fix large constants $M$ and $\ell_0$ which will be chosen later. For $1 \leq i \leq [\frac{n^{1/3}}{2M}]$, we consider the intervals $A_i \subset (\varepsilon, \frac{\pi}{2}-\varepsilon)$, each of length $\frac{(\frac{\pi}{2}-2 \varepsilon)}{n^{1/3}}$ and midpoint $\frac{Mi(\frac{\pi}{2}-2\varepsilon)}{n^{1/3}}$. 
Corresponding to each $A_i$ we define the following random variables. 
\begin{displaymath}
 N_i:=\#\left \{u \in \mathbb{Z}^2: \theta_u \in A_i \text{ and } \boldsymbol{0} \in \Gamma_u^{\theta_u}\right \}.
\end{displaymath}
Then we have the following. 
\begin{displaymath}
    \bigcup_{i=1}^{[\frac{n^{1/3}}{2M}]}\left \{N_i \geq \frac{n^{4/3}}{\ell_0}, n \leq D \leq 2n \right \} \subset \left \{N \geq \frac{n^{4/3}}{\ell_0} \right \}.
\end{displaymath}
For $1 \leq i \leq [\frac{n^{1/3}}{2M}]$, let
\begin{displaymath}
    \mathcal{I}_i:=\left \{N_i \geq \frac{n^{4/3}}{\ell_0}, n \leq D \leq 2n \right \}.
\end{displaymath}
Applying inclusion exclusion principle on the union of $\mathcal{I}_i$ and considering the first two terms (Proposition \ref{pr:bonfe}) we have the following lower bound.
\begin{equation}
\label{inclusion_exclusion}
    \mathbb{P}\left (N \geq \frac{n^{4/3}}{\ell_0} \right) \geq \sum_{i=1}^{[\frac{n^{1/3}}{2M}]}\mathbb{P}(\mathcal{I}_i)-\sum_{1 \leq i < j \leq [\frac{n^{1/3}}{2M}]} \mathbb{P}(\mathcal{I}_i \cap \mathcal{I}_j).
\end{equation}
First we consider the first sum. We fix $1 \leq i \leq [\frac{n^{1/3}}{2M}].$ Let us consider the line segment $V_M$ with $\lceil4Mn^{2/3}\rceil$ many vertices on $\mathcal{L}_0$ with midpoint $\boldsymbol{0}$.
For $v \in V_M$, define
\begin{displaymath}
 N_i^v:=\#\left \{ u \in \mathbb{Z}^2: \theta_u \in A_i \text{ and } v \in \Gamma_u^{\theta_u} \right \}.
\end{displaymath}
\begin{align*}
    &\lceil4Mn^{2/3}\rceil\mathbb{P}\left(\left \{N_i \geq \frac{n^{4/3}}{\ell_0}, n \leq D \leq 2n \right \} \right)
    = \sum_{v \in V_M}\mathbb{P}\left(\left \{N_i^v \geq \frac{n^{4/3}}{\ell_0}, n \leq D_v \leq 2n \right \} \right)\\&=\mathbb{E}(\widehat{N_i}),
\end{align*}
where 
\begin{displaymath}
    \widehat{N_i}:=\sum_{v \in V_M} \mathbbm{1}_{\left \{N_i^v \geq \frac{n^{4/3}}{\ell_0}, n \leq D_v \leq 2n\right \}}.
\end{displaymath}
\begin{figure}[!ht]
    \includegraphics[width=13 cm]{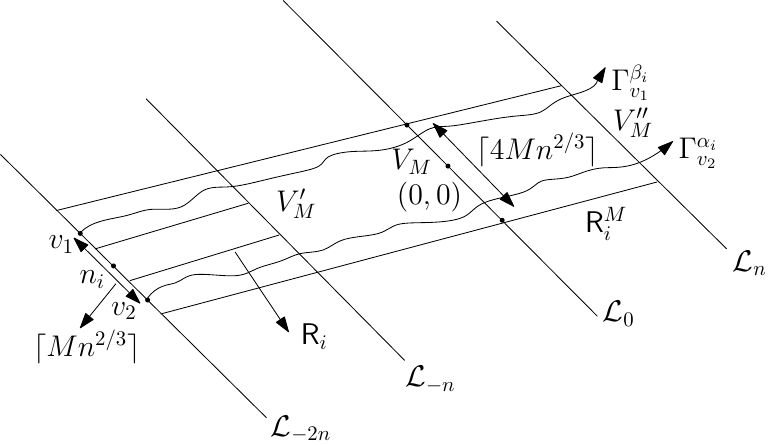}
    \caption{To find lower bound for the first sum in \eqref{inclusion_exclusion} we fix an $i$. Using Proposition \ref{transversal_fluctuation_of_semi_infinite_geodesic} we can fix $M$ large enough so that with high probability $\Gamma_{v_1}^{\beta_i}$ and $\Gamma_{v_2}^{\alpha_i}$ are contained in $\mathsf{R}_i^M$ and do not intersect $\mathsf{R}_i$. On this event any car starting from $\mathsf{R}_i$ and taking angle in $A_i$ will intersect $\mathcal{L}_0$ on $V_M$. Further, using \cite[Theorem 3.10]{BHS18} depending on this $M$ we can choose $\ell_0$ large enough so that with large probability there are at most $\ell_0$ many distinct points on $V_M$ carrying a geodesic starting from $V_M'$ and ending at $V_M''$. Also, with positive probability there are at least $n^{4/3}$ many vertices $u \in \mathsf{R}_i$ such that $\theta_u \in A_i$. The last event is independent of the first two events. Hence, combining all these we get on a positive probability event there exists $v \in V_M$ such that there are at least $\frac{n^{4/3}}{\ell_0}$  many $u$ with $\theta_u \in A_i$ and $v \in \Gamma_u^{\theta_u}$. Now, applying an averaging argument and summing over all $i$ give a lower bound for the first sum in \eqref{inclusion_exclusion}.}
    \label{fig: first_sum_figure} 
\end{figure}
We will show on a positive probability event $\widehat{N_i} \geq 1.$ As before let $\alpha_i, \beta_i, \theta_i$ denote the left end point, right end point, mid point of $A_i$ respectively. Consider $n_i$ which is the intersection point of $\mathcal{L}_{-2n}$ and $y=\tan(\theta_i)x.$ Consider the line segment $\widetilde{V_M}$ with $\lceil Mn^{2/3}\rceil$ many vertices on $\mathcal{L}_{-2n}$ with midpoint $n_i$. Let $v_1$ (resp.\ $v_2$) denote the end points of $\widetilde{V_M}$ with $\psi(v_1) \leq \psi(v_2).$ Further consider the following two parallelograms (see Figure \ref{fig: first_sum_figure}). Let $n_i'$ (resp.\ $n_i''$) denote the intersection point of $y=\tan (\theta_i) x$ with $\mathcal{L}_{-n}$ (resp.\ $\mathcal{L}_n$).
\begin{itemize}
\item $\mathsf{R}_i$ is the parallelogram whose two opposite pair of sides are line segments with $\lceil \frac{n^{2/3}}{4} \rceil$, vertices and they lie on $\mathcal{L}_{-2n}$ (resp.\ $\mathcal{L}_{-n}$) with midpoint $n_i$ (resp.\ $n_i'$).
     \item $\mathsf{R}_i^M$ is the parallelogram whose two opposite pair of sides are line segments with $\lceil 4Mn^{2/3} \rceil$ many vertices , they lie on $\mathcal{L}_{-2n}$ (resp.\ $\mathcal{L}_{n}$) with midpoint $n_i$ (resp.\ $n_i''$).
     \end{itemize}
We define the event $\mathsf{TF}_i$ as follows.
\begin{itemize}
    \item $\mathsf{TF}_i$ is the event that $\Gamma_{v_1}^{\beta_i}$ and $\Gamma_{v_2}^{\alpha_i}$ are contained between the longer sides of $\mathsf{R}_i^M$ and $\Gamma_{v_1}^{\beta_i}$ and $\Gamma_{v_2}^{\alpha_i}$ do not enter $\mathsf{R}_i$.
\end{itemize}
As $(\beta_i-\alpha_i)$ is of order $\frac{1}{n^{1/3}}$, the event $\mathsf{TF}_i$ will imply for sufficiently large $M$ either $\Gamma_{v_1}^{\beta_i}$ or $\Gamma_{v_2}^{\alpha_i}$ will have transversal fluctuation more than $\frac{M}{4}n^{2/3}$. Using Proposition \ref{transversal_fluctuation_of_semi_infinite_geodesic} we can choose $M$ (depending only on $\varepsilon$) large enough and $n$ sufficiently large enough (depending on $M$) so that the following happens.
\begin{displaymath}
    \mathbb{P}(\mathsf{TF}_i) \geq 0.99.
\end{displaymath}
On $\mathsf{TF}_i$, for any $u \in \mathsf{R}_i$ with $\theta_u \in A_i, v \in \Gamma_{u}^{\theta_u}$, for some $v \in V_M.$ Now we define an equivalence class. Let $n_i''$ denote the intersection point of $y=\tan (\theta_i) x$ with $\mathcal{L}_n$. Consider the two line segments $V_M'$ (resp.\ $V_M''$)  with $\lceil 4Mn^{2/3} \rceil$ vertices on $\mathcal{L}_{-n}$ (resp.\ $\mathcal{L}_n$) with midpoint $n_i'$ (resp.\ $n_i''$). For $u,v \in V_M'$ and $u',v' \in V_M'',$ we say $(u,v) \sim (u',v')$ if $\Gamma_{u,v}(0)=\Gamma_{u',v'}(0).$ Let $E_i$ denote the number of equivalence classes. Using Proposition \ref{coalescence_theorem}, we can choose $\ell_0$ large enough (depending on $M$) so that for all sufficiently large $n$
\begin{displaymath}
 \mathbb{P}\left (E_i \leq \frac{\ell_0}{4} \right ) \geq 0.99.
\end{displaymath}
Now, we consider the following random variable.
\begin{displaymath}
    Y_n:=\sum_{u \in \mathsf{R}_i}\mathbbm{1}_{\{\theta_u \in A_i\}}.
\end{displaymath}
Note that for large $n$, $\mathbb{E}(Y_n) \geq \frac{n^{4/3}}{4}$ and $\Vv(Y_n)=n \lceil \frac{n^{2/3}}{4}\rceil\left(\frac{1}{n^{1/3}}-\frac{1}{n^{2/3}} \right)$. By the Berry-Esseen inequality (Proposition \ref{p: Berry-Essen}) we have the following for all $n$ and a constant $c$
\begin{displaymath}
    |\mathbb{P}(\{Y_n-\mathbb{E}(Y_n)\geq 0\})-1/2| \leq \frac{c}{(n^{4/3}-n)^{3/2} n^{5/6}}.
\end{displaymath}
Hence, for all sufficiently large $n$ we have 
\begin{displaymath}
    \mathbb{P}\left (\left\{Y_n-\frac{n^{4/3}}{4} \geq 0\right \} \right ) \geq \frac{1}{4}.
\end{displaymath}
Finally we observe,
\begin{equation}
\label{event_on_which_there_is_1_good_vertex}
    \mathsf{TF}_i \cap \left \{E_i \leq \frac{\ell_0}4\right\} \cap \left \{Y_n-\frac{n^{4/3}}{4} \geq 0 \right \} \subset \{\widehat{N_i} \geq 1\}.
\end{equation}
Further, the event $\left \{Y_n-\frac{n^{4/3}}{4} \geq 0 \right \}$ is independent of $\mathsf{TF}_i \cap \left \{E_i \leq \frac{\ell_0}{4} \right \}$. So, combining all the above for sufficiently large $n,$ there exists $c_1>0$ ($c_1$ is a fixed constant, does not depend on anything) such that 
\begin{displaymath}
    \mathbb{E}(\widehat{N_i}) \geq \mathbb{P}\left(\mathsf{TF}_i \cap \left \{E_i \leq \frac{\ell_0}{4} \right \} \cap \left \{Y_n-\frac{n^{4/3}}{4} \geq 0 \right \}\right) \geq c_1.
\end{displaymath}
Hence,
\begin{equation}
\label{first_sum_lower_bound}
    \mathbb{P}\left(\left \{N_i \geq \frac{n^{4/3}}{\ell_0}, n \leq D \leq 2n \right \} \right) \geq \frac{c_1}{4Mn^{2/3}}.
\end{equation}
Now, we consider the second sum in \eqref{inclusion_exclusion}. We fix $1\leq i <j \leq [\frac{n^{1/3}}{2M}],$  both large enough. 
We recall the events $\mathcal{E}_i$ defined in the proof of lower bound of Theorem \ref{depth_bounds_theorem}. 
Clearly, for all $1 \leq i \leq [\frac{n^{1/3}}{2M}]$
\begin{displaymath}
    \mathcal{I}_i \subset \mathcal{E}_i.
\end{displaymath}
So, as we did in the proof of lower bound of Theorem \ref{depth_bounds_theorem}, from \eqref{first_sum_lower_bound}, \eqref{second_sum_upper_bound} and \eqref{inclusion_exclusion} we have 
\begin{displaymath}
    \mathbb{P}(N \geq \frac{n^{4/3}}{\ell_0}) \geq \frac{c_1}{8M^2 n^{1/3}}-\sum_{1 \leq i <j \leq \frac{n^{1/3}}{2M}}\frac{C_2 \log (M|i-j|)}{M^2|i-j|^2n^{2/3}}.
\end{displaymath}
As argued before, choosing $M$ sufficiently large completes the proof of Theorem \ref{number of cars upper bound} lower bound (with \(n\leftrightarrow n^{4/3}\), as before). \qed
\section{Distance to find a road with large number of cars}\label{Distance to find a road with large number of cars}
In this section we analyse how far one needs to go from origin to see a road with large number of cars. Recall that for a fixed constant $\ell_0$ we define 
\begin{displaymath}
 T^{\ell_0}_n:=\min \left \{ |\psi(v)| : v \in \mathcal{L}_0, N_v \geq \frac{n^{4/3}}{\ell_0} \right \}.
\end{displaymath}
\subsection{Proof of Theorem \ref{road_lower_bound}} Let $I$ denote the line segment with $\lfloor n^{1/3} \rfloor$ many vertices on $\mathcal{L}_0$ with midpoint $\boldsymbol{0}$. We divide the interval $(\varepsilon, \frac{\pi}{2}-\varepsilon)$ as we did in the proof of lower bound of Theorem \ref{number of cars upper bound}. i.e., we fix large constants $M$ and $\ell_0$ which will be chosen later. For $1 \leq i \leq [\frac{n^{1/3}}{2M}]$, we consider the intervals $A_i \subset (\varepsilon, \frac{\pi}{2}-\varepsilon)$, each of length $\frac{(\frac{\pi}{2}-2 \varepsilon)}{n^{1/3}}$ and midpoint $\frac{Mi(\frac{\pi}{2}-2\varepsilon)}{n^{1/3}}$.

Let us define the following events. For $1 \leq i \leq [\frac{n^{1/3}}{2M}]$ we define 
\begin{itemize}
    \item $\mathcal{K}_i:=\{ \exists w \in I$ such that there are $\frac{n^{4/3}}{\ell_0}$ many vertices $v$ with $-2n \leq \phi(v) \leq -n, \theta_v \in A_i$ and $w \in \Gamma_v^{\theta_v}$\}.
\end{itemize}
Clearly,
\begin{displaymath}
 \bigcup_{i=1}^{[\frac{n^{1/3}}{2M}]} \mathcal{K}_i \subset \{T^{\ell_0}_n\leq n^{1/3}\}.
\end{displaymath}
Same as before, considering the first two terms after applying the inclusion-exclusion principle (Proposition \ref{pr:bonfe}) we have
\begin{equation}
\label{inclusion_exclusion_1}
    \mathbb{P}(T^{\ell_0}_n\leq n^{1/3}) \geq \sum_{i=1}^{[\frac{n^{1/3}}{2M}]}\mathbb{P}(\mathcal{K}_i)-\sum_{1 \leq i < j \leq [\frac{n^{1/3}}{2M}]} \mathbb{P}(\mathcal{K}_i \cap \mathcal{K}_j).
\end{equation}
We first consider the first sum. We fix $1 \leq i \leq [\frac{n^{1/3}}{2M}].$ We will again apply an averaging argument. But this time we will average over disjoint intervals containing $\lfloor n^{1/3}\rfloor$ vertices each. Precisely, we do the following. We consider $V_M$, line segment with $\lfloor 4Mn^{1/3} \rfloor \lfloor n^{1/3} \rfloor$ many vertices on $\mathcal{L}_0$ with midpoint $\boldsymbol{0}$. We partition it into $\lfloor 4Mn^{1/3} \rfloor $ many mutually disjoint sub-intervals each with $\lfloor n^{1/3} \rfloor$ many vertices inside $V_M$. i.e., consider the mutually disjoint intervals $I_j$ each with $\lfloor n^{1/3} \rfloor$ many vertices. For each of these intervals we define the following events. 
\begin{itemize}
    \item $\mathcal{K}_i^j:=\{ \exists w \in I_j$ such that there are $\frac{n^{4/3}}{\ell_0}$ many vertices $v$ with $-2n \leq \phi(v) \leq -n, \theta_v \in A_i$ and $w \in \Gamma_v^{\theta_v}\}.$
\end{itemize}
We have
\begin{displaymath}
    \mathbb{P}(\mathcal{K}_i)=\mathbb{P}(\mathcal{K}_i^j).
\end{displaymath}
So,
\begin{equation}
\label{averaging_equation}
    \lfloor 4Mn^{1/3}\rfloor \mathbb{P}(\mathcal{K}_i)=\sum_{1 \leq j \leq \lfloor 4Mn^{1/3} \rfloor}\mathbb{P}(\mathcal{K}_i^j)=\mathbb{E}\left(\sum_{1 \leq j \leq \lfloor 4Mn^{1/3} \rfloor} \mathbbm{1}_{\mathcal{K}_i^j} \right).
\end{equation}
We will show, on a positive probability set, for large $n$ the following holds. 
\begin{displaymath}
    \sum_{1 \leq j \leq \lfloor 4Mn^{1/3} \rfloor} \mathbbm{1}_{\mathcal{K}_i^j} \geq 1. 
\end{displaymath}
But observe that, precisely this we have shown in the proof of Theorem \ref{number of cars upper bound} lower bound. Specifically, in \eqref{event_on_which_there_is_1_good_vertex} we have shown that on the event $\mathsf{TF}_i \cap \left \{E_i \leq \frac{\ell_0}{4} \right\} \cap \{Y_n-\frac{n^{4/3}}{4} \geq 0\}$ there is at least one vertex $w \in V_M$ such that there are $\frac{n^{4/3}}{\ell_0}$ many $v \in \mathbb{Z}^2$ with $-2n \leq \phi(v) \leq -n, \theta_v \in A_i$ and $w \in \Gamma_v^{\theta_v}.$ So, we have
\begin{displaymath}
    \mathsf{TF}_i \cap \{E_i \leq \ell_0\} \cap \left \{Y_n-\frac{n^{4/3}}{4} \geq 0 \right \} \subset \left \{\sum_{1 \leq j \leq \lfloor 4Mn^{1/3}\rfloor} \mathbbm{1}_{\mathcal{K}_i^j} \geq 1 \right \}.
\end{displaymath}
Hence, there exists $c>0$ such that for sufficiently large $n$ we have
\begin{displaymath}
    \mathbb{E}\left(\sum_{1 \leq j \leq \lfloor 4Mn^{1/3} \rfloor} \mathbbm{1}_{\mathcal{K}_i^j} \right) \geq c.
\end{displaymath}
So, from \eqref{averaging_equation} we have for large $n$, and for some $c>0$,
\begin{displaymath}
    \mathbb{P}(\mathcal{K}_i) \geq \frac{c}{Mn^{1/3}}.
\end{displaymath}
Now, we consider the second sum in \eqref{inclusion_exclusion_1}. We fix $1 \leq i < j \leq [\frac{n^{1/3}}{2M}]$. Same as before we consider the following events. For $1 \leq i \leq [\frac{n^{1/3}}{2M}],$ we define 
\begin{itemize}
    \item $\mathcal{M}_i:=\{ \exists u \in \mathbb{Z}^2, \gamma \in A_i$ such that $\phi(u)=-n$, and $\exists w \in I$ such that $w \in \Gamma_u^{\gamma}\}.$
\end{itemize}
Clearly,
\begin{displaymath}
    \mathcal{K}_i \subset \mathcal{M}_i.
\end{displaymath}
So, we will work with $\mathcal{M}_i$ from now on. We will again apply an averaging argument in both space and time directions. But this time in the space direction we will average over disjoint intervals of length $n^{1/3}.$ Precisely, we fix $1 \leq i< j \leq [\frac{n^{1/3}}{2M}]$. As before recall that $\theta_i$ and $\theta_j$ are midpoints of $A_i$ and $A_j$. Consider the parallelograms $W_{i,j}$ depending on $\theta_i$ and $\theta_j$ as defined in the proof of Theorem \ref{depth_bounds_theorem} lower bound with the smaller side length $n^{1/3}\cdot n^{1/3}$. We now fix $m$ such that $-\lceil \frac{n}{100k_{i,j}} \rceil \leq m \leq \lceil \frac{n}{100k_{i,j}} \rceil$ and consider $W_{i,j} \cap \mathcal{L}_m$. We consider these line segments of length $n^{1/3}\cdot n^{1/3}$ and partition them into line segments with $\lfloor n^{1/3} \rfloor$ many vertices. So, we get, for large \(n\), at least $\frac{c n^{4/3}}{100k_{i,j}}$ many mutually disjoint line segments each with $\lfloor n^{1/3} \rfloor $ many vertices for some $c>0$ uniformly bounded away from 0. We will denote the line segments by $I_{k,m}$, if the line segment lies on $\mathcal{L}_m$ and $k$ is the index to specify the partition in the space direction on $\mathcal{L}_m$. For all $-\lceil \frac{n}{100 k_{i,j}} \rceil \leq m \leq \lceil \frac{n}{100k_{i,j}}\rceil$ and $1 \leq k \leq 2\lfloor n^{1/3} \rfloor$ and for all $1 \leq i \leq [\frac{n^{1/3}}{2M}]$ we define the following events.
\begin{itemize}
    \item $\mathcal{M}_i^{k,m}:=\{ \exists u \in \mathbb{Z}^2, \gamma \in A_i$ such that $\phi(u)=-n+m$, and $\exists w \in I_{k,m}$ such that $w \in \Gamma_u^{\gamma}\}$.
    \end{itemize}
    Clearly,
    \begin{displaymath}
        \mathbb{P}\left(\mathcal{M}_i \cap \mathcal{M}_j\right)=\mathbb{P}\left(\mathcal{M}_i^{k,m} \cap \mathcal{M}_j^{k,m}\right).
    \end{displaymath}
    So, we have 
    \begin{align*}
        \frac{c n^{4/3}}{100k_{i,j}}\mathbb{P}\left(\mathcal{M}_i \cap \mathcal{M}_j\right)\le\sum_{\substack{-\lceil \frac{n}{100k_{i,j}} \rceil \leq m \leq \lceil \frac{n}{100k_{i,j}} \rceil\\ 1 \leq k \leq 2\lfloor n^{1/3}\rfloor}}\mathbb{P}\left(\mathcal{M}_i^{k,m} \cap \mathcal{M}_j^{k,m}\right)\\ =\mathbb{E}\left(\sum_{\substack{-\lceil \frac{n}{100 k_{i,j}} \rceil \leq m \leq \lceil \frac{n}{100k_{i,j}} \rceil\\ 1 \leq k \leq 2\lfloor n^{1/3} \rfloor}} \mathbbm{1}_{\mathcal{H}_i^{k,m} \cap \mathcal{H}_j^{k,m}} \right) \leq \mathbb{E}(N_{i,j}).
    \end{align*}
    Recall that we had defined $N_{i,j}$ in the proof of lower bound in Theorem \ref{depth_bounds_theorem} and from \eqref{expected_number_inside_parallelogram_bound} we have that there exists constant $C>0$ such that for large $n$
    \begin{displaymath}
        \frac{c n^{4/3}}{100k_{i,j}}\mathbb{P}\left(\mathcal{M}_i \cap \mathcal{M}_j\right) \leq \frac{C n \log |k_{i,j}|}{k_{i,j}^3}.
    \end{displaymath}
Combining all the above and from \eqref{inclusion_exclusion_1} we have that there are $C,c>0$ (depending only on $\varepsilon$)
\begin{displaymath}
 \mathbb{P}(T^{\ell_0}_n\leq n^{1/3}) \geq \sum_{i=1}^{[\frac{n^{1/3}}{2M}]}\frac{c}{Mn^{1/3}}-\sum_{1 \leq i <j \leq [\frac{n^{1/3}}{2M}]}\frac{C\log (M|i-j|)}{M^2|i-j|^2n^{1/3}}.
\end{displaymath}
Hence, again we have for some constant $C,c>0$ we have
\begin{displaymath}
 \mathbb{P}(T^{\ell_0}_n\leq n^{1/3}) \geq \frac{c}{M^2}-\frac{C}{M^3}.
\end{displaymath}
Finally, we can choose $M$ large enough so that the right side is positive. This completes the proof. \qed

\section{Empirical analysis of road networks}\label{sc:elev}

		This section describes our approach for empirically validating theoretical predictions about the alignment of road networks with geodesic paths and the statistics that follows. A significant basis for our model is the hypothesis that in hilly terrain, real-world road network configurations substantially align with geodesic shortest paths when elevation variations are considered. Our road network model leverages the Last Passage Percolation (LPP) model for its theoretical tractability in predicting road network structures. However, we acknowledge that road networks more realistically conform to the dynamics of First Passage Percolation (FPP) models, despite their lesser tractability.

		The interconnection between LPP and FPP models within the Kardar-Parisi-Zhang (KPZ) universality class offers a theoretical justification for this approach. This KPZ linkage underpins our assumption that insights derived from the LPP framework can provide predictive insights for FPP models.

Our validation consists of two parts. First we verify that, on the input side of the model, FPP in a hilly environment indeed predicts where roads are built, this is done in Section \ref{sc:hillval}. We then turn to compare the scaling behavior in Theorems \ref{high_probability_event} and \ref{road_lower_bound} to actual road layout and traffic statistics in Section \ref{sc:comptheo}.

\subsection{Road networks and elevation data}\label{sc:hillval}

		To ensure the relevance and accuracy of FPP models in capturing the essence of real-world road layouts, this section is dedicated to validating this assumption. By examining the alignment of road networks against geodesic paths influenced by elevation data, we aim to validate the predictive capability of FPP models and, in turn, our LPP variant.
		
	\subsubsection{Data acquisition and processing}
	
	Our study leverages high-resolution elevation data sourced from the Shuttle Radar Topography Mission (SRTM), accessed through the Earth Explorer platform \cite{earthexplorer}. The SRTM dataset provides comprehensive global elevation data. From this dataset, we extract elevation data encoded into image files, which serve as input in building an edge weighted graph.

	\subsubsection{Graph construction and path analysis}

In this study, we leveraged a computational framework found within the GitHub repository  \cite{statmechmodels} to facilitate our analysis of road networks within a designated study area. Utilizing the code from this repository enabled us to construct a weighted graph that models a square lattice segment of the terrain under investigation. Each vertex within this graph represents a precise point on the terrain, and the weights assigned to the edges are calculated based on both the elevation differences and the distances between the vertices at the endpoints. 
 
	Specifically, if $\Delta$ denotes the grid spacing between elevation samples and $h$ is a function expressing the elevation at a vertex then the edge weight on the edge $\{x,y\}$ of the graph is taken to be 
 $$w_{\{x,y\}} = \sqrt { \Delta^2 + (h(x)-h(y))^2  }.$$

    Moreover, we have found that in some cases more consistent results can be achieved between the predicted paths and the observed roads if we treat $\Delta$ as a tuneable parameter.  
      
	\subsubsection{Figures}

	We provide Figures \ref{fig: FigureLong}, \ref{fig:FigureLowRes} and \ref{fig:FigureHighRes} to illustrate the comparison between actual road segments and theoretical geodesic paths derived from our graph model. Each figure shows a map provided by ESRI World Topo Map \cite{esriworldtopo2024} with a geodesic path, calculated based on elevation data and shortest path algorithms, overlaid on selected road segments. The intent behind these figures is to visually highlight where and how actual road networks align with or deviate from the calculated geodesic paths. This direct comparison serves as evidence for the hypothesis that road layouts tend to follow geodesic paths, considering elevation.

\begin{figure}[ht]
 \begin{center}
    \includegraphics[width=12 cm]{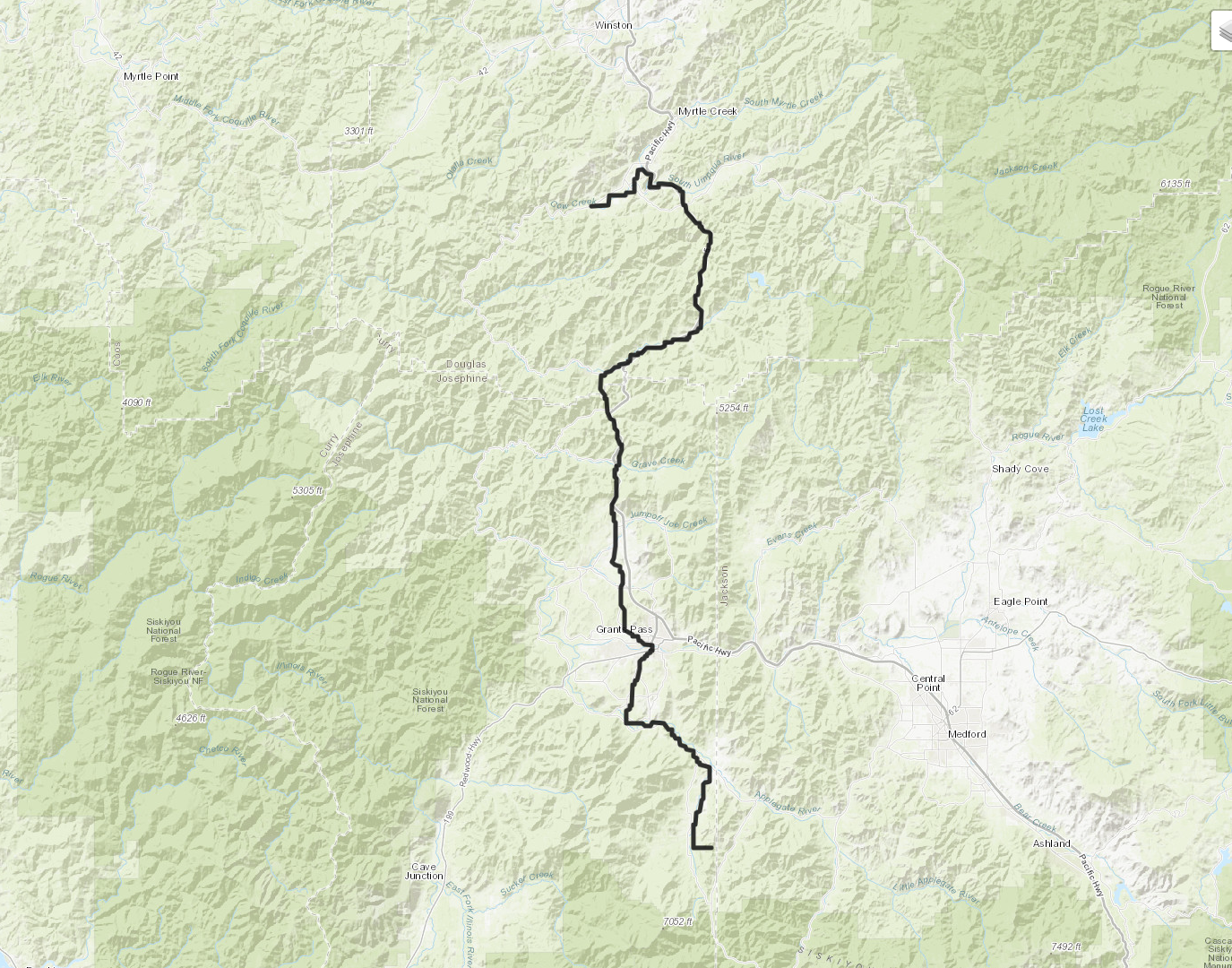}
\end{center}
    \caption{A path computed between coordinates $(42.9251086, -123.4220748)$ and $(42.2017416, -123.2366819)$.}
    \label{fig: FigureLong}
\end{figure}

\begin{figure}[ht]
    \centering
    \begin{minipage}[b]{0.45\textwidth}
        \includegraphics[width=\textwidth]{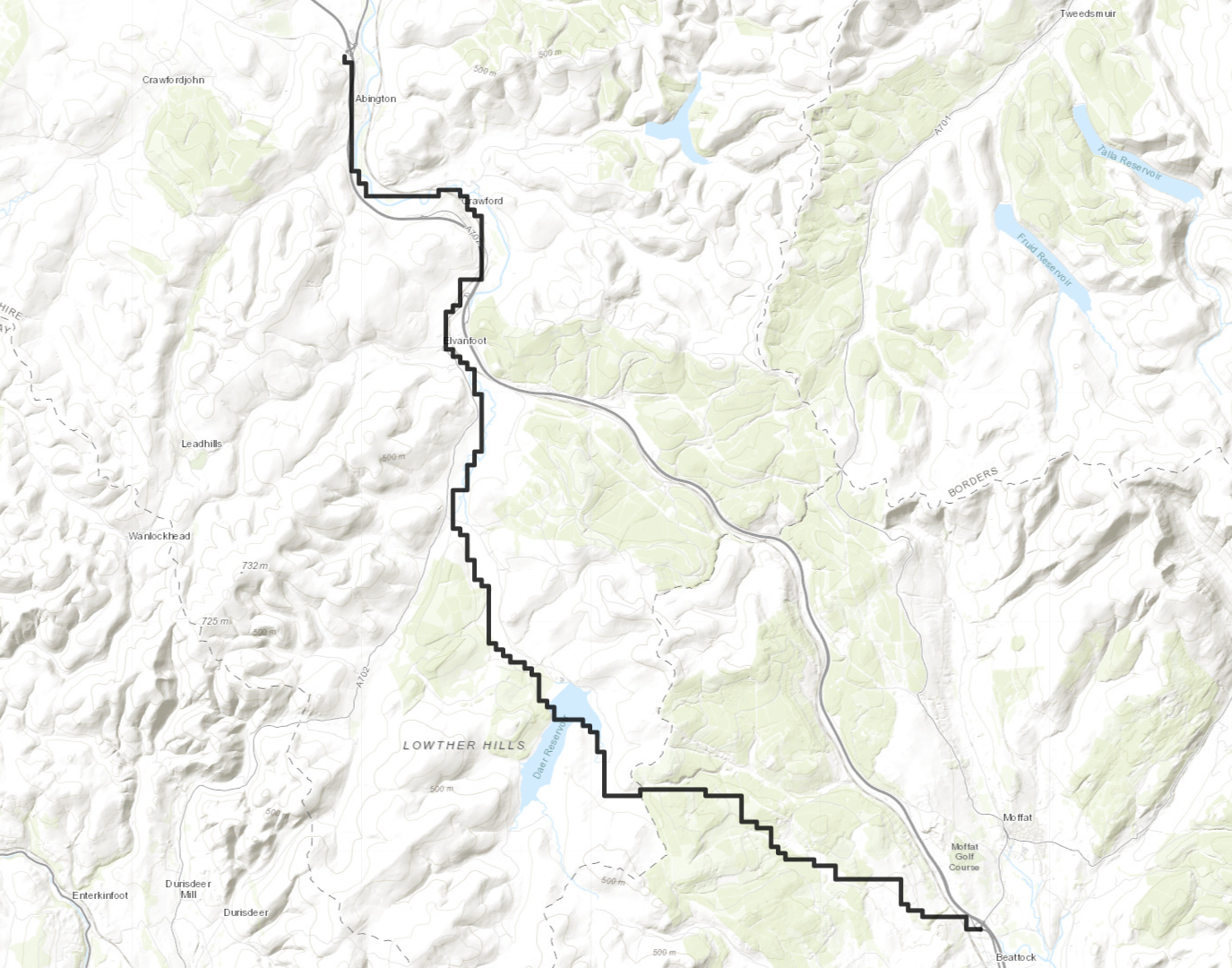}
        \caption{A path computed between $(55.5028, -3.6964)$ and  $(55.3124,-3.4540)$. We see an apparent disagreement between portions of the road and geodesic caused by low resolution lattice spacing.}
        \label{fig:FigureLowRes}
    \end{minipage}
    \hfill 
    \begin{minipage}[b]{0.45\textwidth}
        \includegraphics[width=\textwidth]{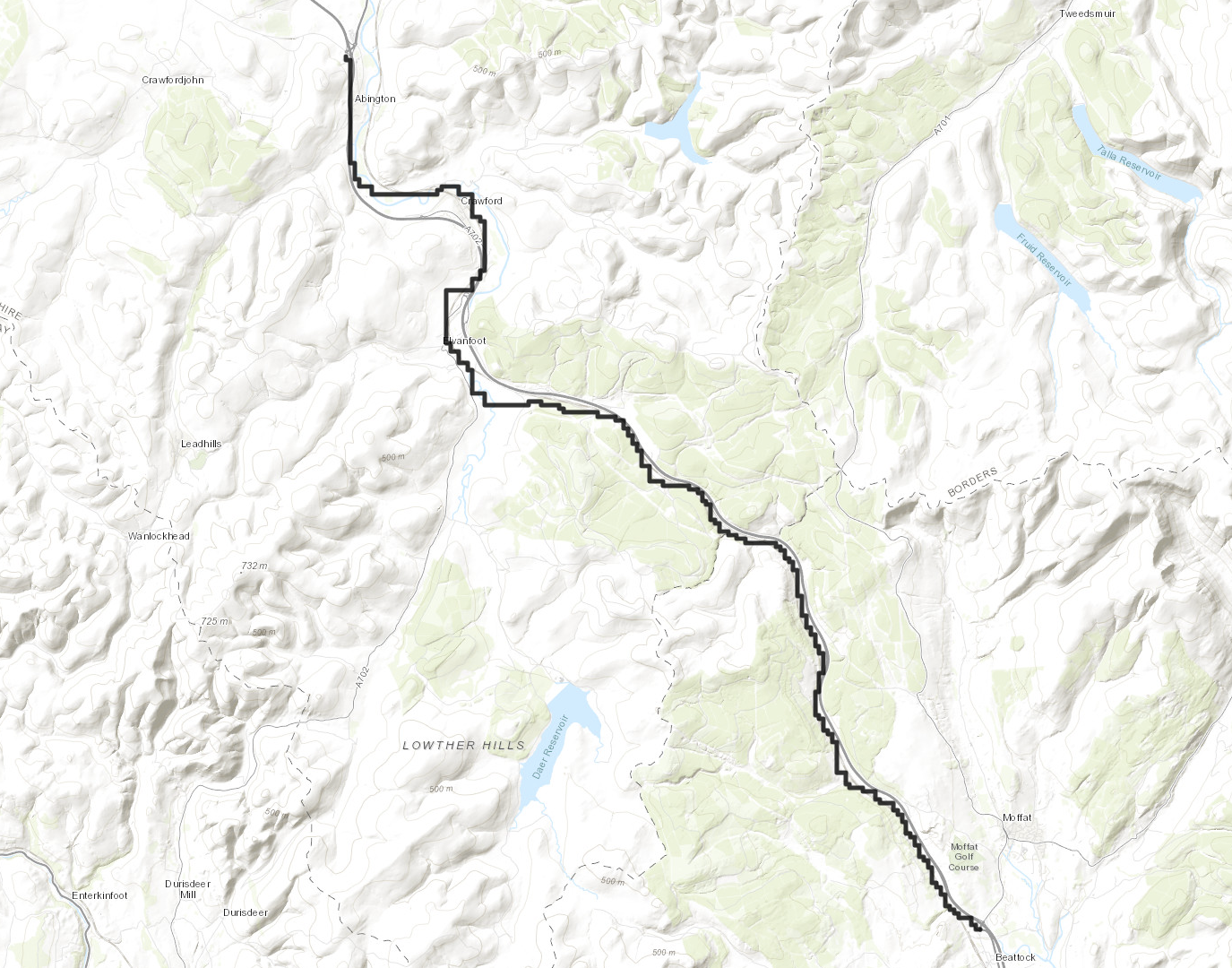}
        \caption{A path computed between $(55.5028, -3.6964)$ and  $(55.3124,-3.4540)$. This time we see better agreement when refining lattice scale.}
        \label{fig:FigureHighRes}
    \end{minipage}
\end{figure}

	The figures largely corroborate our hypothesis that real-world road networks tend to follow the theoretical shortest paths deduced from our model, with particular emphasis on elevation data. Nonetheless, certain regions exhibit discrepancies, prompting a discussion of potential causes for these divergences.
	
	\paragraph{Finite Lattice Effects} One factor to consider is the graph model's finite lattice structure, which simulates the continuous real-world terrain with a discrete grid. This approximation introduces potential errors in shortest path calculations, where the model's sensitivity to edge weight value can manifest as notable deviations from expected paths. Such minor inaccuracies in terrain representation, therefore, might explain the observed misalignment between theoretical predictions and actual road configurations. This effect is made apparent between Figures \ref{fig:FigureLowRes} and \ref{fig:FigureHighRes}.
	
	\paragraph{Beyond Shortest Path Considerations} Additionally, the real-world intricacies of road design decision-making extend beyond the simplistic criteria of minimizing path length or elevation variations. Factors such as legal restrictions, environmental conservation, and socio-economic impacts heavily influence road planning, often leading to routes that diverge from those predicted by our shortest path model. Elements like land ownership, ecological preservation areas, and pre-existing structures necessitate alterations to the theoretically optimal paths. These considerations, absent from the shortest path algorithm, might account for any discrepancies between our model's predictions and the existing road networks.
	 
	\paragraph{}
	Nevertheless, the clear alignment observed between the model predictions and actual road layouts, despite some discrepancies, would seem to support the use of the FPP model for characterizing road network structures. Then by the KPZ correspondence mentioned earlier, we can feel more confident in the predictive capability of the LPP variant. We now turn to verifying some of these predictions on actual road networks.

\subsection{Comparison of the model output and real-world traffic data}\label{sc:comptheo}

We use the government traffic dataset from the United Kingdom accessed at \cite{roadtraffic_gov_uk}. This consists of annual average daily traffic flow at count points scattered around the roads. We are interested in investigating the validity of Theorems \ref{high_probability_event} and \ref{road_lower_bound} for this dataset. We will disregard the logarithmic correction in the former. Fix a startpoint and consider distances \(d\) along a spatial direction of LPP. These results say that, with probability separated away from both 0 and 1, we are expected to get at least one road with traffic count proportional to $d^{4}$ within this distance $d$. As opposed to LPP, we consider real road networks to be isotropic, which means that the direction we choose should be irrelevant. We therefore work with East as our preferred direction. The task is to pick points on the map randomly, draw an Easterly line from each, record where this intersects actual roads as well as the traffic volumes carried by these roads.

It is quite difficult to determine an accurate traffic count on a specific intersection point with our Easterly line because the closest count point on that road might be far away from here. It might also happen that there are many junctions between the intersection point and the nearest count point on the road, harming the accuracy of the count point to the traffic count at the intersection of interest. Figure \ref{fig:diffmeas} illustrates the difficulties.
\begin{figure}[ht]
 \includegraphics[width=\textwidth]{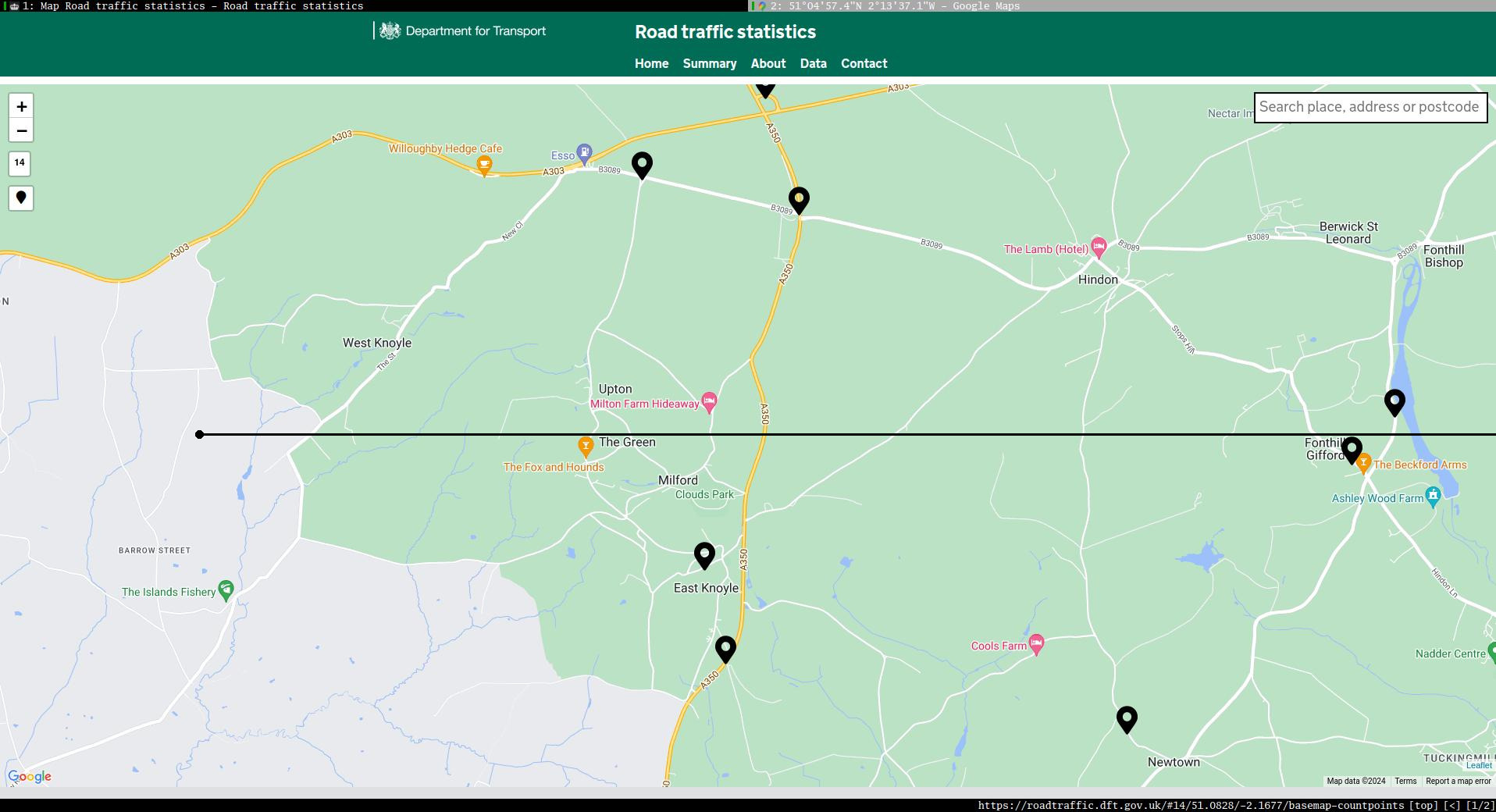}
 \caption{Manual measurement of distances and road traffic. This is the interface of the \emph{Road traffic statistics} webpage, with a startpoint and its Easterly line overlaid by us in black. The pinpoints are the locations of the traffic count points. They can be clicked, and their data is also available to bulk-download. To be recorded are the distances between the startpoint at the left end of the black line and the intersection points of roads with the black line, as well as the traffic carried by those roads.}\label{fig:diffmeas} 
\end{figure}

Careful inspection of the map and of the location of traffic count points allow for such data to be collected with reasonable accuracy, especially for busier roads. We will refer to this as the \emph{manual method}. Smaller lanes are often not measured, we had no choice but to simply leave them out of our analysis.
However, this process is really slow and would not give sufficient volumes of data for proper statistical analysis. Unfortunately, we failed to develop an algorithm that could perform this data collection for us.

Instead, we consider a strip around the Easterly line starting from the chosen random point $x$ and take all the count points that lie inside this strip. We record their distances from \(x\) and the traffic count they possess, and pretend that these are actual locations of intersections between the Easterly line and the roads that they measure. We will refer to this process as the \emph{automated method}.

In the rest of the paper we first calibrate the strip width to give a good match between outcomes of the manual method and the automated method. We do that on a few examples where the manual method was carried out. We then run the automated method on about 2\,000 randomly chosen points on the map, and compare the statistics gained with the statements of our Theorems \ref{high_probability_event} and \ref{road_lower_bound}.

\subsubsection{Calibrating the automated method}

We performed the manual method at four locations of the UK, near Cirencester (51.7076,-1.90897), Gillingham (51.0826,-2.22697), Glasgow (55.606053,-4.0506979), and a point in Wales (52.1577699,-3.5958771). The code to support this analysis, along with additional simulation results is found in a dedicated GitHub repository created by the authors for this project \cite{githublpptrafficmodel}.

While the startpoint should be completely random, large-scale inhomogeneity of the UK and its harmful effects soon became apparent. First, big cities are obviously much denser than rural areas, distorting the measurement. Hence we made sure to avoid startpoints from where Easterly lines would have traced into larger urban areas.

Second, the UK has a number of estuaries that needed to be avoided as they form long natural barriers to road building.

Finally, we found that a dense part in England, roughly in the rectangle Drochester -- Hastings -- Hull -- Blackpool, has a significantly denser road network than the North and the West of the UK. Hence we performed our analysis separately for this South-East rectangle, and the North and West regions.

Of the four locations above for the manual method, Cirencester and Gillingham belong to the denser South-East block, while the Wales and the Glasgow startpoints are in the sparser West and North region.

We then ran the automated method on the same four startpoints, with different strip width. Strip width is very important. If the strip is very wide, then it will contain many countpoints which might wrongly include some high-traffic roads that don't actually intersect the Easterly line. Similarly, if the strip is too narrow, then it might miss some important countpoints of high-traffic roads because those countpoints can lie outside of the strip even if that road intersects the Easterly line.

To check the theoretical predictions, we need to answer whether or not there is at least one road above certain traffic threshold within certain distance of the startpoint we chose. This is equivalent to considering the running maximum of crossing traffic volumes as we move alongside the Easterly line from the startpoint. It is therefore this running maximum which we tried to align between the manual and the automated method. The figures below illustrate this graph on the left hand-side. For better intuition we also included violin plots on the right hand-side, which show the traffic distribution measured in the two methods.

Figures \ref{fig: South_2_1km}, \ref{fig: South_2_3km} and \ref{fig: South_2_5km} refer to the startpoint near Cirencester (51.7076,-1.90897) in the South-East region.
It is noted that sometimes the violin plot takes a peak when there are multiple traffic counts around the same distance. For example, consider the orange line in Figure \ref{fig: South_2_5km}. From the left side, it shows that there is a new maximum that is observed around 60\,km after a peak at 40\,km. But on the right side, traffic distribution shows that the traffic around 60\,km is lower than the traffic at a distance of 40\,km. This happens because multiple traffic counts of nearly the same values are observed near 40\,km. The plots indicate that the appropriate strip width is between 3\,km and 5\,km.

Similarly, also in the South-East region Figures \ref{fig: South_3_1km}, \ref{fig: South_3_3km} and \ref{fig: South_3_5km} are from the startpoint (51.0826,-2.22697) near Gillingham. Figures \ref{fig: West_1_3km}, \ref{fig: West_1_8km} and \ref{fig: West_1_10km} are the plots with the startpoint near Wales (52.1577699,-3.5958771) in the sparser, West region of UK. Finally, Figure \ref{fig:North_10} is the plot with the startpoint near Glasgow, which belongs to the North region of UK. When taking the manual measurement for this plot, we get a high-traffic road within a distance of 10km, but when considering the automated measurement of a strip width of 10\,km, we encounter another, much higher-traffic road, and these remain the busiest crossings until the distance of 150\,km. This makes the runing maximum very different between the manual and the automated method. However, if we neglect the first few mismatched measurements, we get a good comparison. Figure \ref{fig:North_10} illustrates it all. As the points are randomly chosen, rarely bad points like this one are expected to occur, and we do not have an efficient way to eliminate this issue.

After observing these plots, we have concluded that 3\,km would be the ideal strip width for the South-East region and 10\,km for the North and West regions of the UK.

\subsubsection{Statistical analysis}

We were now ready to deploy the automated method in large scale. As mentioned above, the plots (Figures \ref{fig: South_2_1km} - \ref{fig:North_10}) indicate that the statistics are different between the South-East and the North and West regions of the UK. We therefore divide the UK into these two regions for our statistical analysis.
Around 2000 random startpoints were considered inside the UK, and those points are at least 80\,km horizontally away from big cities (like London, Birmingham, etc.), big rivers or oceans to avoid distorting effects to the measurements.
For a fixed strip width and distance $d$, we observe how many points $x$ have traffic counts greater than $k$, and calculate the relative frequency by dividing by the total number of points, where $k$ runs from $1000$ to $100000$. Figures \ref{fig: Stat_South_3km} and \ref{fig: Stat_West_10km} show the traffic threshold encountered with relative frequency between $[0.49, 0.51]$
as a function of the distance travelled on the Easterly line.
\begin{figure}[ht]
    \centering
    \includegraphics[width=12 cm]{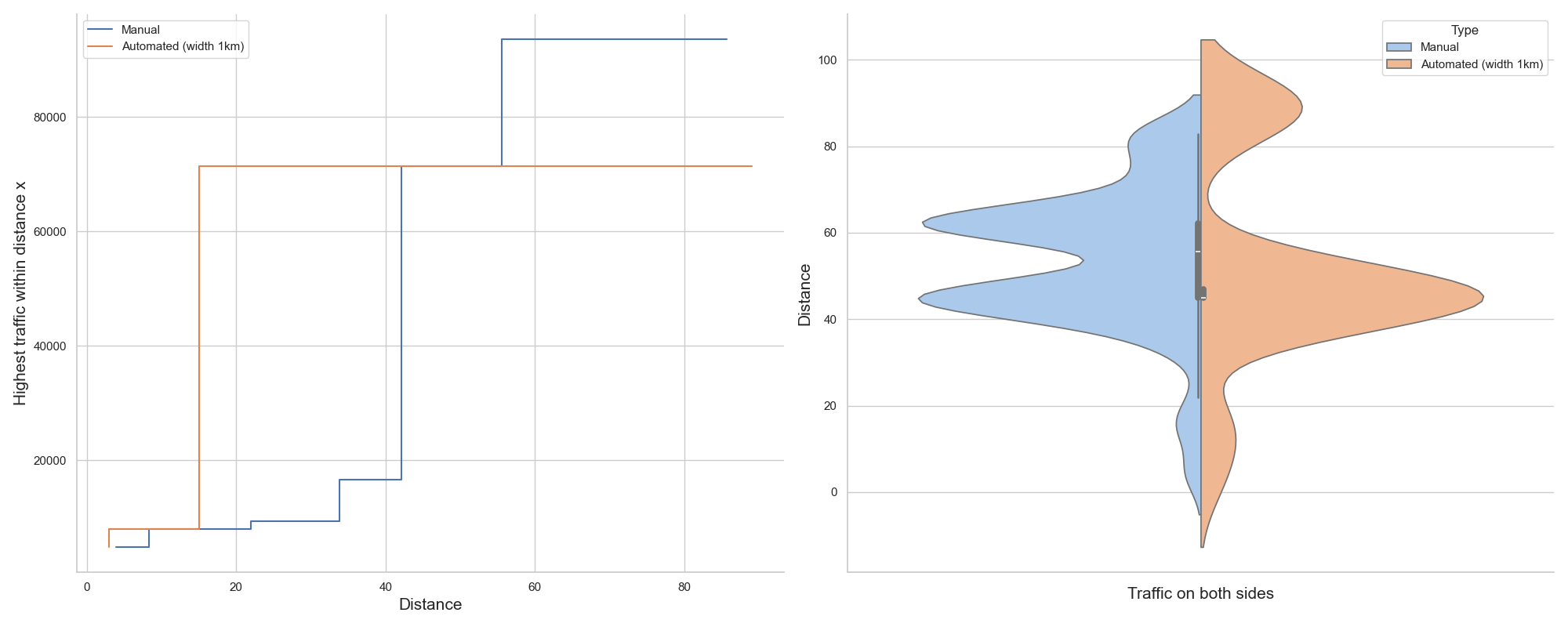}
    \caption{Busiest traffic from the startpoint (51.7076,-1.90897) (to the left) and violin plot (to the right) with strip width 1\,km.}
    \label{fig: South_2_1km} 
    \centering
    \includegraphics[width=12 cm]{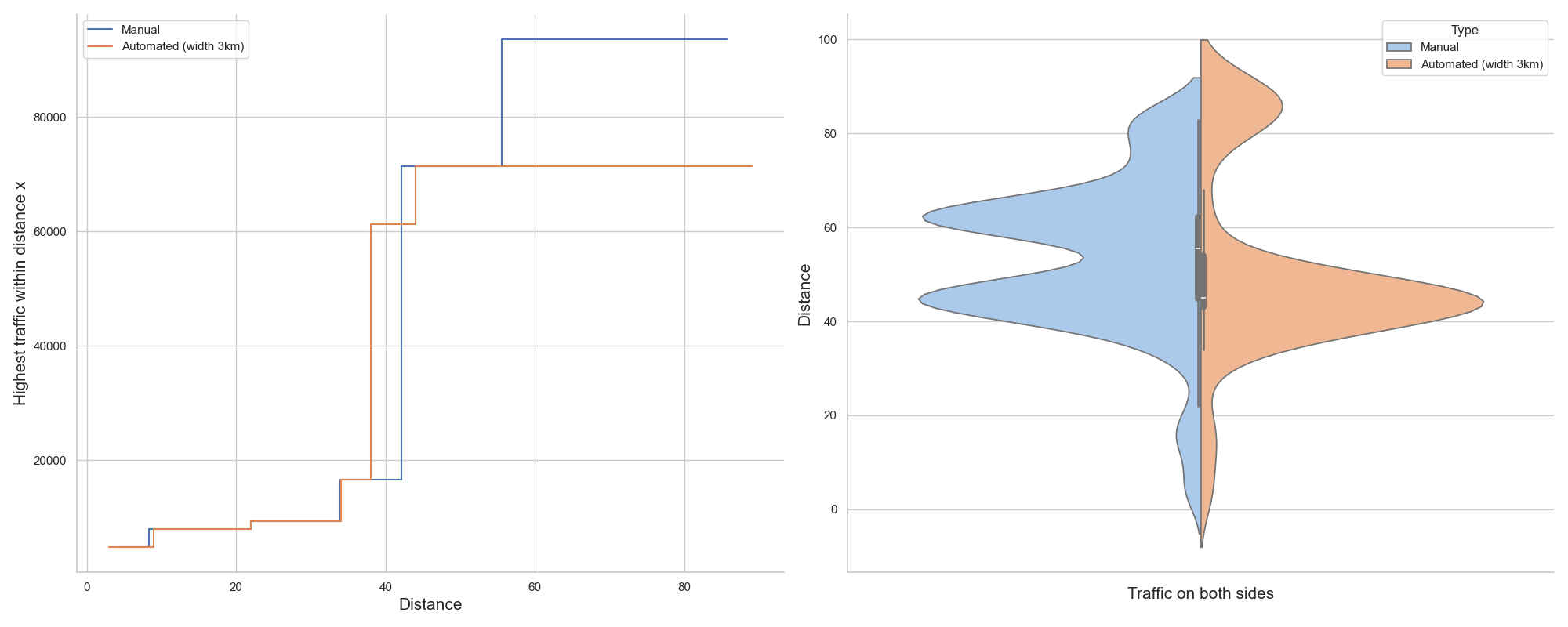}
    \caption{Busiest traffic from the startpoint (51.7076,-1.90897) (to the left) and violin plot (to the right) with strip width 3\,km.}
    \label{fig: South_2_3km} 
    \centering
    \includegraphics[width=12 cm]{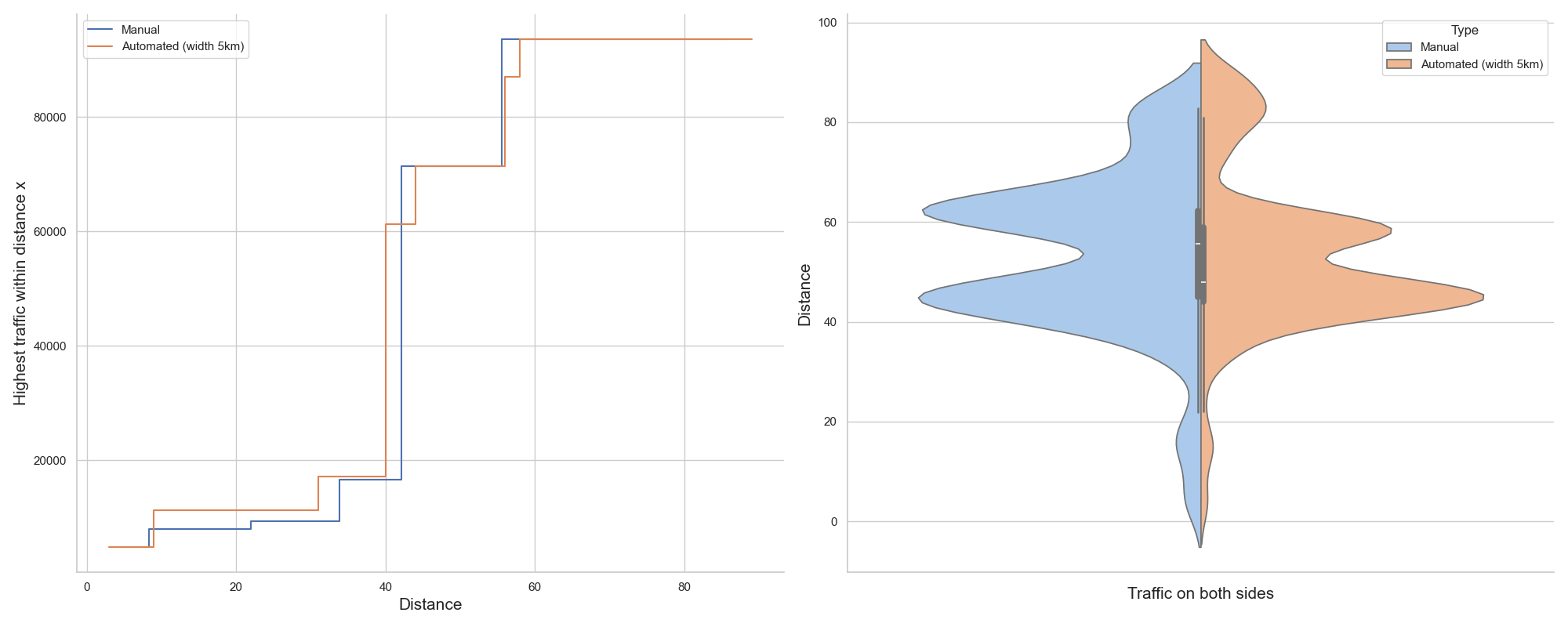}
    \caption{Busiest traffic from the startpoint (51.7076,-1.90897) (to the left) and violin plot (to the right) with strip width 5\,km.}
    \label{fig: South_2_5km} 
\end{figure}
\begin{figure}[ht]
    \centering
    \includegraphics[width=12 cm]{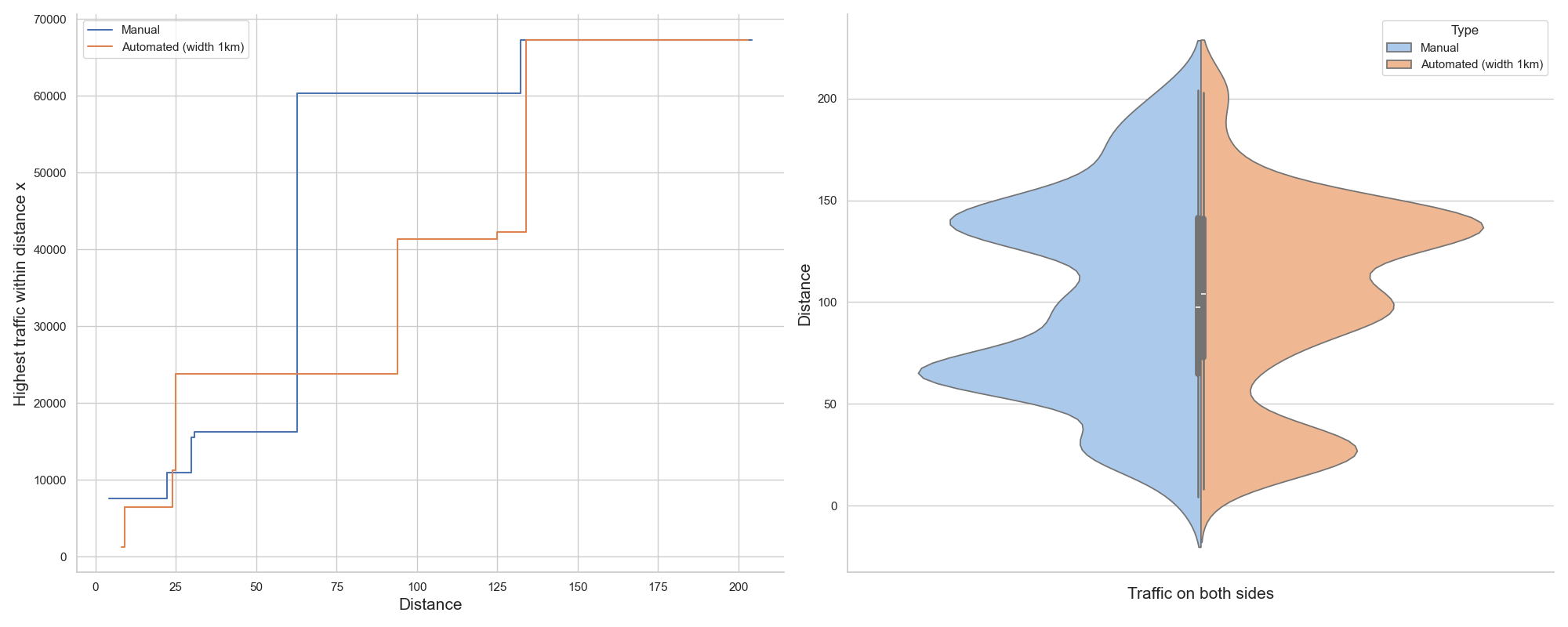}
    \caption{Busiest traffic from the startpoint (51.0826,-2.22697) (to the left) and violin plot (to the right) with strip width 1\,km.}
    \label{fig: South_3_1km} 
    \centering
    \includegraphics[width=12 cm]{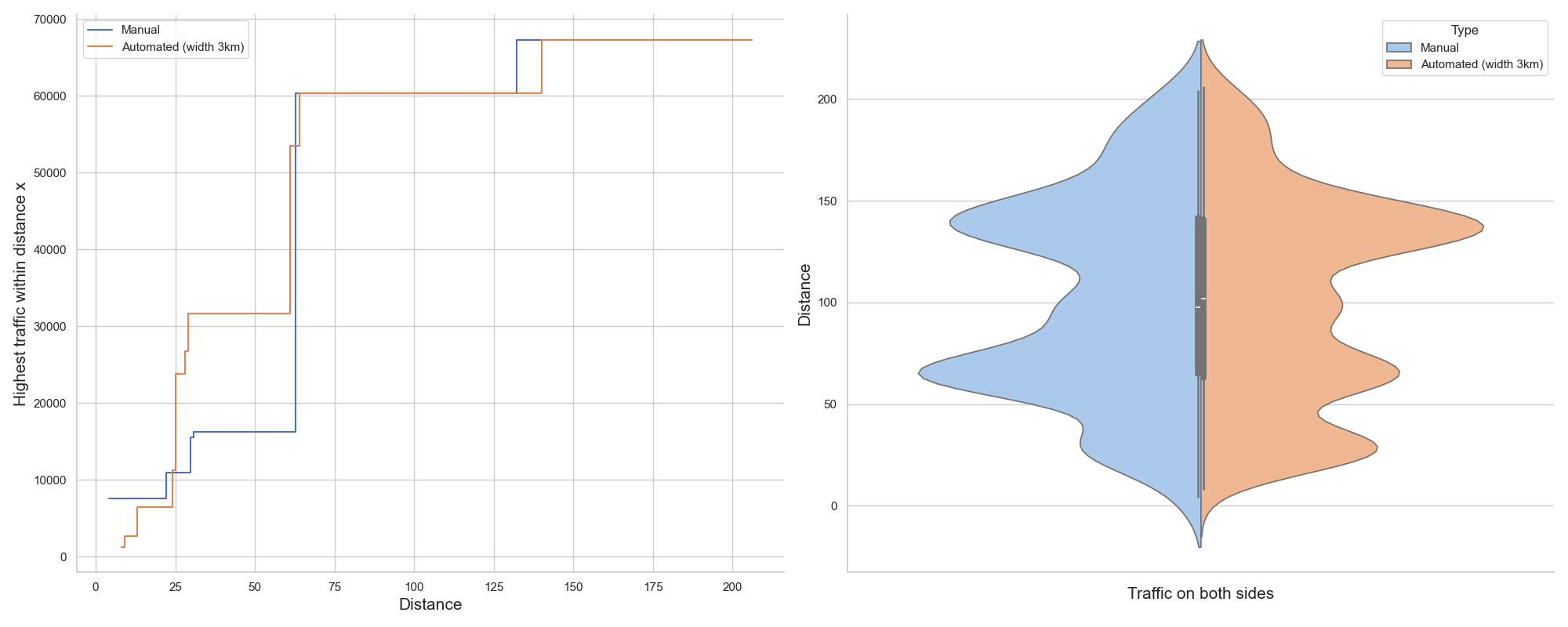}
    \caption{Busiest traffic from the startpoint (51.0826,-2.22697) (to the left) and violin plot (to the right) with strip width 3\,km.}
    \label{fig: South_3_3km} 
    \centering
    \includegraphics[width=12 cm]{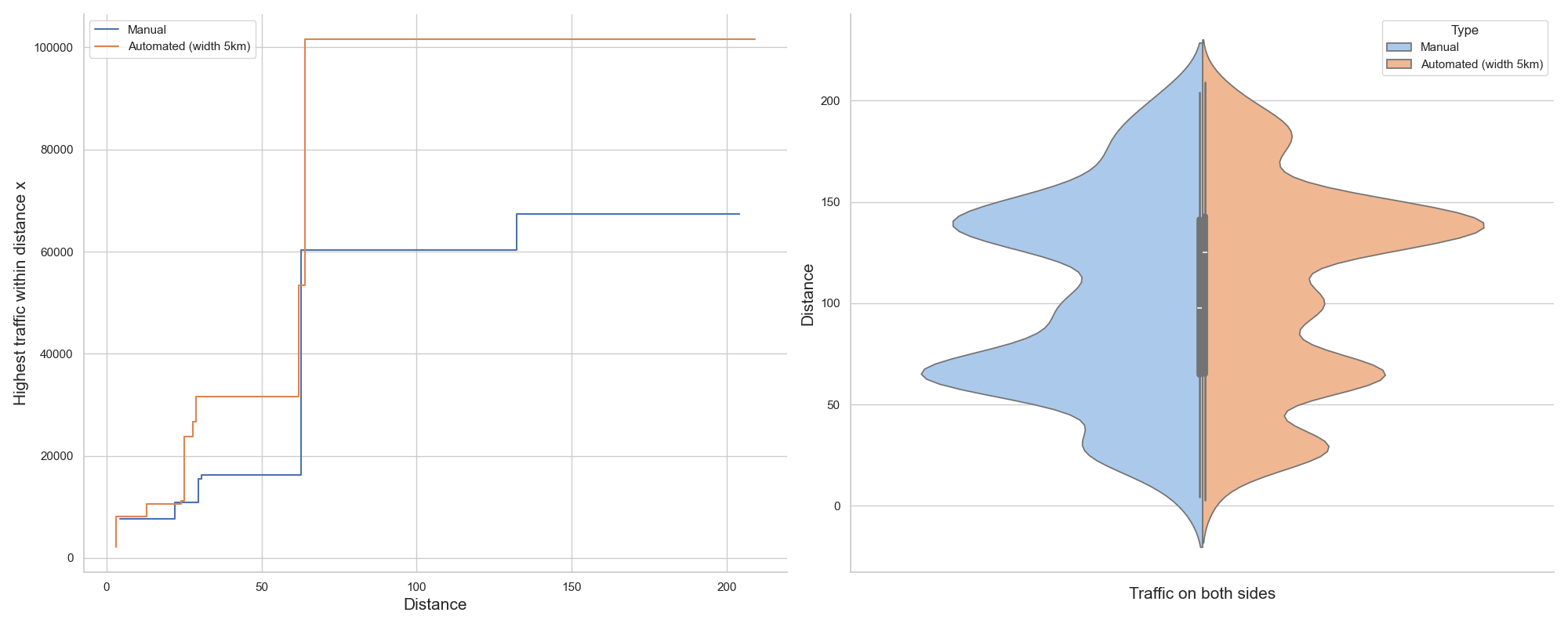}
    \caption{Busiest traffic from the startpoint (51.0826,-2.22697) (to the left) and violin plot (to the right) with strip width 5\,km.}
    \label{fig: South_3_5km} 
\end{figure}
\begin{figure}[ht]
    \centering
    \includegraphics[width=12 cm]{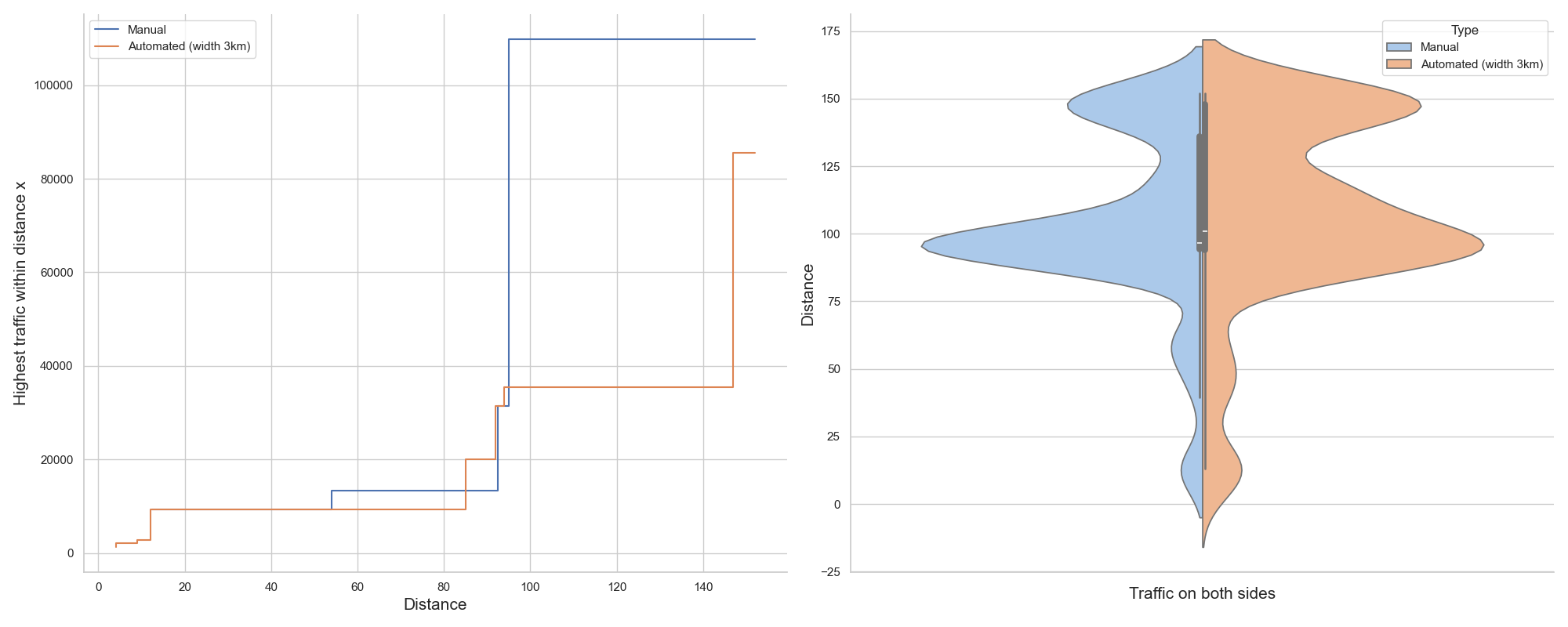}
    \caption{Busiest traffic from the startpoint (52.1577699,-3.5958771) (to the left) and violin plot (to the right) with strip width 3\,km.}
    \label{fig: West_1_3km}
    \centering
    \includegraphics[width=12 cm]{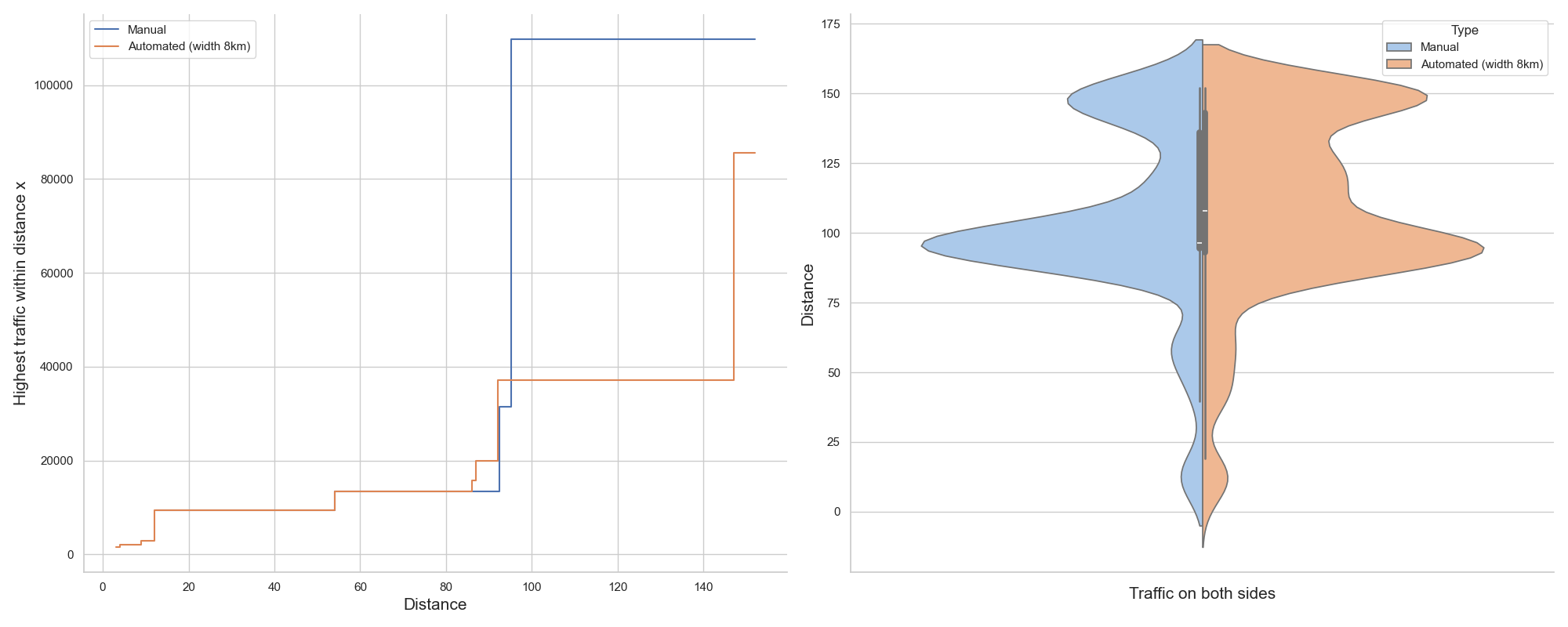}
    \caption{Busiest traffic from the startpoint (52.1577699,-3.5958771) (to the left) and violin plot (to the right) with strip width 8\,km.}
    \label{fig: West_1_8km}
    \centering
    \includegraphics[width=12 cm]{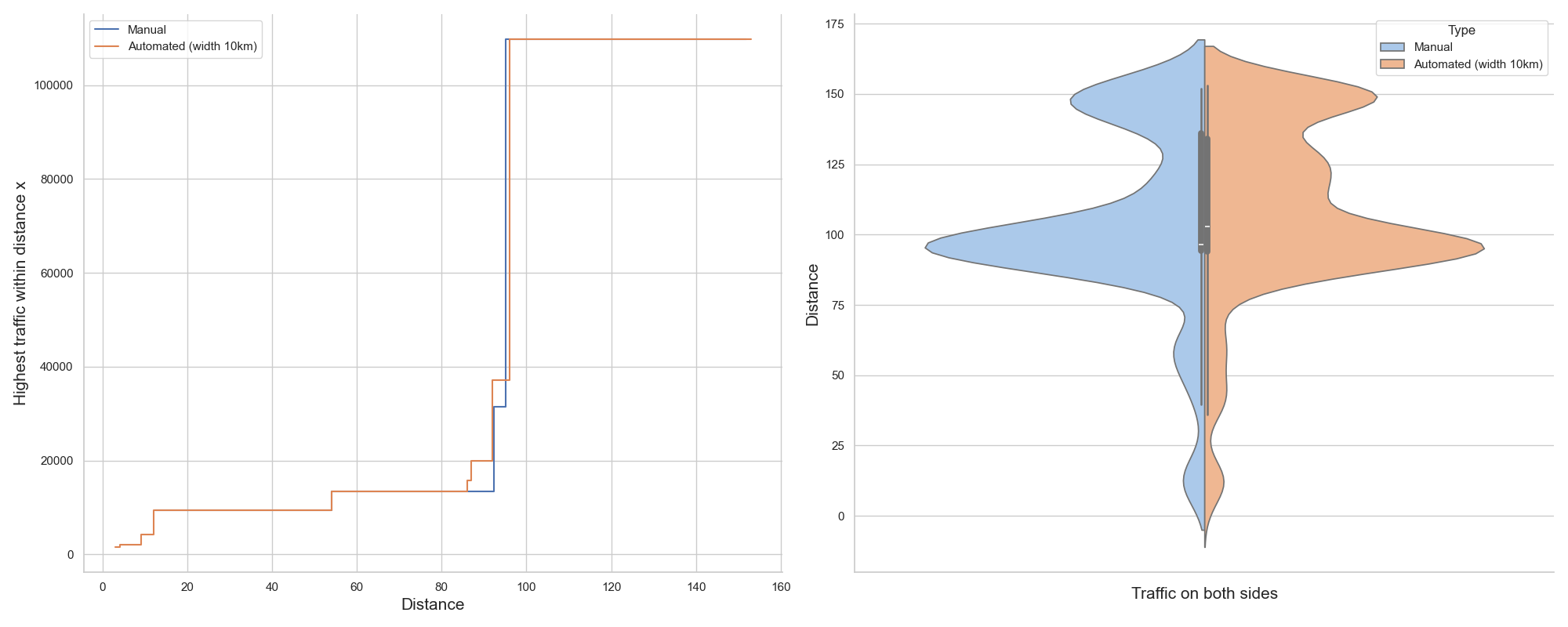}
    \caption{Busiest traffic from the startpoint (52.1577699,-3.5958771) (to the left) and violin plot (to the right) with strip width 10\,km.}
    \label{fig: West_1_10km}
\end{figure}
\begin{figure}[ht]
    \centering
    \includegraphics[width=12 cm]{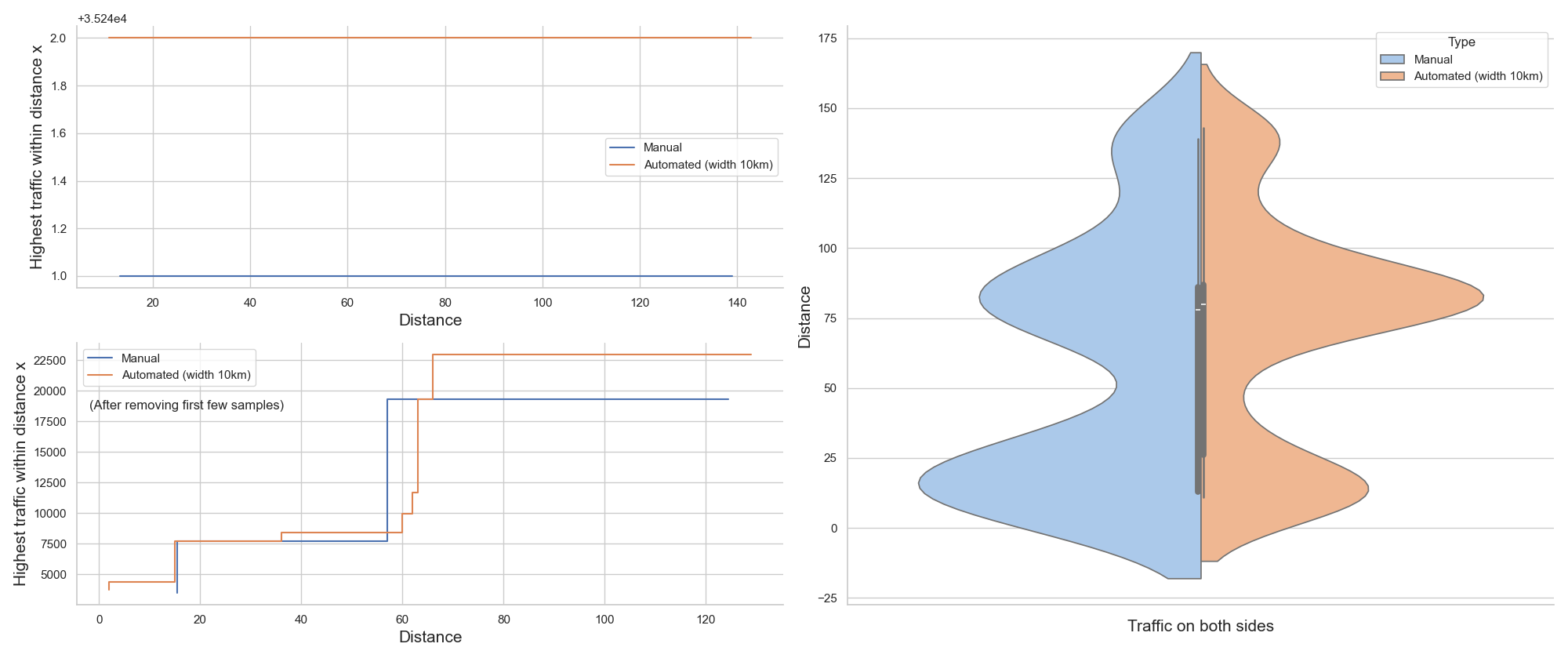}
    \caption{To the left, the busiest traffic from the startpoint (55.606053,-4.0506979) and to the right, the violin plot with a strip width 10\,km. On the left side, the top figure indicates that the busiest traffic was captured within the first 10\,km distance by the automated method, but this was a false count point that shouldn't have been included. Inaccuracies like this are expected from time to time. However, the bottom figure shows a much better fit after removing the closest few traffic points.}
    \label{fig:North_10}
\end{figure}

Figure \ref{fig: Stat_South_3km} is the plot when the South-East region is considered with strip width 3\,km, whereas Figure \ref{fig: Stat_West_10km} for the West and North region with the strip width 10\,km. By Theorem $\ref{high_probability_event}$ and Theorem \ref{road_lower_bound}, with positive probability, we should expect $Cd^4$ graph from these plots, where $C$ is constant, but $C$ depends on the region where we are doing the analysis. We saw similar plots when relative frequency thresholds different from $[0.49, 0.51]$ were tried. The plots do not resemble a quartic function, which we explain next.
\begin{figure}[ht]
    \centering
    \includegraphics[width=10 cm]{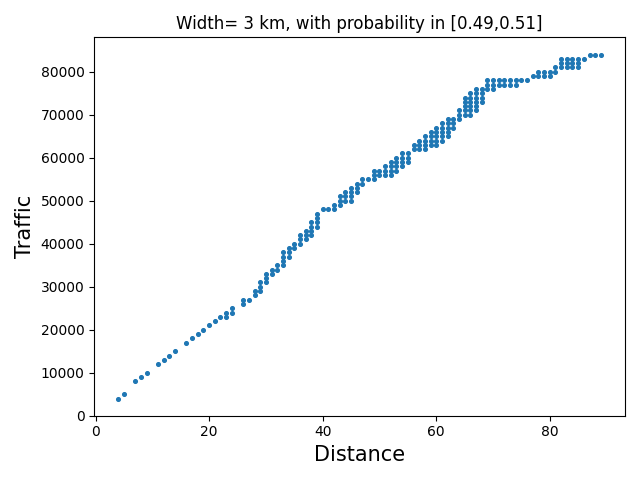}
    \caption{Statistical plot for the South-East region of the UK}
    \label{fig: Stat_South_3km}
\vspace{2cm}
    \centering
    \includegraphics[width=10 cm]{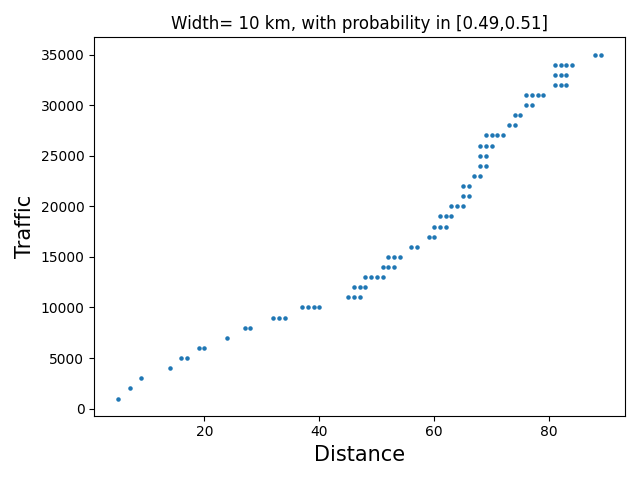}
    \caption{Statistical plot for the West region of the UK}
    \label{fig: Stat_West_10km}
\end{figure}

\subsubsection{Limitations of our analysis}

There are obvious simplifications where our model and statistical procedure do not align with reality. The first of these is that geography is not an i.i.d.\ environment. Long valleys are carved by rivers, introducing obvious correlations in the landscape.

The second observation is the difficulties with extracting geographic information on intersection points between Easterly lines and actual roads. We discussed this in details earlier.

None of these two issues appear to fundamentally change the main ideas of this paper to the extent seen by our UK-wide analysis. Instead, we believe that the main problem is a combination of driving distances and the nature of the road statistics data that we worked with. As apparent from Figure \ref{fig:diffmeas}, small roads are most often not measured, while the bigger, well-documented roads tend to be 20\dots30\,km apart. Combine this with the typical driving distances of 10\dots15\,km for a UK car trip, and add to the picture that two LPP geodesics towards the same direction, started from distance \(d\) apart, need to travel in the order of \(d^{3/2}\) distance before they have a reasonable chance to coalesce. The typical UK road trip is simply too short to have a good opportunity to join one of the well-measured, big roads. Hence, we are missing the upper tail of the traffic count distribution that would show the sharp growth of a \(d^4\) function. We suspect that more detailed measurements of smaller roads would match the predictions better but such data just do not seem to be available.

Further observations distancing our model from real road networks can be found in Solon et al.\ \cite{sol_bun_chu_kar_opti_paths_poly}. As already mentioned in the introduction, they found that for road networks the environment is best described as having a heavy-tailed distribution, which is not the case for our framework of Exponential LPP. They also found long-range positive correlations in road density, spanning across hundreds of kilometers. Without traffic counts we do not have a meaningful way to define road density in our model as we assume that cars start their journeys from each point of \(\mathbb Z^2\). As for busier roads in our model, we would guess negative correlations rather than positive ones. This might be due to a busy road representing a stretch of preferable environment, which might attract other traffic hence make busy roads less likely nearby; this could be a question for future mathematical research.

\section*{Acknowledgements}

The authors thank Riddhipratim Basu for useful discussions throughout the project. We also thank for discussions with Yuri Bakhtin, Atal Bhargava, Sanchari Goswami and Aquib Molla in the early phases of this work, and Janosch Ortmann later. We are also grateful for thorough reading of the manuscript and helpful comments on improvements by anonymous reviewers.

MB was partially supported by the Engineering and Physical Sciences Research Council EP/W032112/1 Standard Grant of the UK. SB was supported by scholarship from National Board for Higher Mathematics (NBHM) (ref no: 0203/13(32)/2021-R\&D-II/13158). DH was partially supported by National Science Foundation grant DMS-2054559. KD was supported by the Prime Minister's Research Fellowship PM/PMRF-22-18686.03 (PMRF ID: 0202965) from the Ministry of Education of India. This project was initiated at the International Centre for Theoretical Sciences (ICTS), Bengaluru, India during the program ``First-passage percolation and related models'' in July 2022 (code: ICTS/fpp-2022/7), the authors thank ICTS for the hospitality.

\smallskip\noindent
All data used in this study are publicly available at \cite{roadtraffic_gov_uk} and \cite{earthexplorer}.

\section*{Appendix: Auxiliary Estimates}
 Here we state the estimates that we have used in the proofs. For completeness, we first start with a simple special case of Bonferroni's inequality.
 \begin{proposition}\label{pr:bonfe}
  Let \(k\ge2\) and \(A_1,\,A_2,\,\dots A_k\) be any events in a probability space. Then
  \[
   \mathbb P\Bigl(\bigcup_{i=1}^kA_i\Bigr)\ge\sum_{i=1}^k\mathbb P(A_i)-\sum_{1\le i<j\le k}\mathbb P(A_i\cap A_j).
  \]
 \end{proposition}
 \begin{proof}
  We show this by induction. When \(k=2\), the equality is the base case of the inclusion-exclusion formula. Let us now assume the statement for \(k\), and write
  \[
   \begin{aligned}
    \mathbb P\Bigl(\bigcup_{i=1}^{k+1}A_i\Bigr)
    &=\mathbb P\biggl(\Bigl(\bigcup_{i=1}^kA_i\Bigr)\cup A_{k+1}\biggr)\\
    &=\mathbb P\Bigl(\bigcup_{i=1}^kA_i\Bigr)+\mathbb P(A_{k+1})-\mathbb P\biggl(\Bigl(\bigcup_{i=1}^kA_i\Bigr)\cap A_{k+1}\biggr)\\
    &=\mathbb P\Bigl(\bigcup_{i=1}^kA_i\Bigr)+\mathbb P(A_{k+1})-\mathbb P\Bigl(\bigcup_{i=1}^k(A_i\cap A_{k+1})\Bigr).
   \end{aligned}
  \]
  We can now apply the inductive hypothesis on the probability of first union, while Boole's inequality on the second. With the minus sign of the display these point in the same direction:
  \[
   \begin{aligned}
    \mathbb P\Bigl(\bigcup_{i=1}^{k+1}A_i\Bigr)
    &\ge\sum_{i=1}^k\mathbb P(A_i)-\sum_{1\le i<j\le k}\mathbb P(A_i\cap A_j)+\mathbb P(A_{k+1})-\sum_{i=1}^k\mathbb P(A_i\cap A_{k+1})\\
    &=\sum_{i=1}^{k+1}\mathbb P(A_i)-\sum_{1\le i<j\le{k+1}}\mathbb P(A_i\cap A_j)
   \end{aligned}
  \]
  as required by collapsing the summations.
 \end{proof}
 Next, the statement of Berry-Esseen inequality.

 \begin{proposition}[Berry-Essen inequality, {\cite[XVI.5.]{Feller_book}}]
\label{p: Berry-Essen}Let $X_1,X_2,\dots$ be i.i.d. with $\mathbb{E}(X_i)=0, \mathbb{E}X_i^2=\sigma^2, \mathbb{E}|X_i|^3=\rho < \infty.$ If $F_n(x)$ is the distribution of $\frac{X_1+X_2+\dots+X_n}{\sigma \sqrt n}$ and $\mathcal{N}(x)$ is the standard normal distribution, then for all $n$ and $x$
\[
\left \vert F_n(x)-\mathcal{N}(x) \right \vert \leq \frac{3 \rho}{\sigma^3 \sqrt n}.
\]
\end{proposition}
Next we state the estimates from exponential LPP. We start with transversal fluctuation of semi-infinite geodesics.

Let $\Gamma^\alpha$ ( resp.\ $\mathcal{J}^\alpha$) denote the semi-infinite geodesic (resp.\ the straight line) in the direction $\alpha$ starting from $\boldsymbol{0}$ and $\Gamma^{\alpha}(T)$ (resp.\ $\mathcal{J}^{\alpha}(T)$) denote the intersection points of $\Gamma^\alpha$ (resp.\ $\mathcal{J}^\alpha$) with $\mathcal{L}_{T}.$ We have the following.
\begin{proposition}\cite[Proposition 1.5]{BBB23}
\label{transversal_fluctuation_of_semi_infinite_geodesic}
For $\varepsilon>0$ there exist $C_1,c_1>0$ such that for all $T >0, \ell >0, n \geq 1$ and for all $\alpha \in (\varepsilon, \frac{\pi}{2}-\varepsilon)$ we have
\begin{enumerate}[label=(\roman*), font=\normalfont]
    \item 
            $\mathbb{P}(|\psi(\Gamma^{\alpha}(T))-\psi(\mathcal{J}^{\alpha}(T))|\ge  \ell T^{2/3}) \le C_1e^{-c_1 \ell^3}$,
            \item 
            $\mathbb{P}( \sup \{ |\psi(\Gamma^{\alpha}(t))-\psi(\mathcal{J}^{\alpha}(t))|: 0 \leq t \leq T)\} \geq \ell T^{2/3}) \leq C_1e^{-c_1\ell^3}$.
            \end{enumerate}
\end{proposition}
Next we state a proposition about the fact that geodesics tend to coalesce. 

For $k \in \mathbb{Z}$, let $L_n$ (resp. $L_n^*$) denote the line segment on $\mathcal{L}_0$ (resp. $\mathcal{L}_{2n}$) of length $2\ell^{1/32}n^{2/3}$ with midpoint $\boldsymbol{0}$ (resp. $\boldsymbol{n}_k=(n-kn^{2/3},n+kn^{2/3})$). For $u,u' \in L_n$ and $v,v' \in L_n^*$ we say that $(u,v) \sim (u',v')$ if the geodesics $\Gamma_{u,v}$ and $\Gamma_{u',v'}$ coincide between the lines $\mathcal{L}_{n/3}$ and $\mathcal{L}_{2n/3}$. It is easy to see that $\sim$ is an equivalence relation. Let $M_n^k$ denote the number of equivalence classes.
\begin{proposition}\cite[Proposition 1.6]{BBB23}
\label{coalescence_theorem}
For $\psi < 1$ there exist $C,c > 0$ such that for all $k$ with $|k|+\ell^{1/32} < \psi n^{1/3},$ all $\ell < n^{0.01}$ sufficiently large and all $n \in \mathbb{N}$ sufficiently large we have 
\begin{equation}
\label{tail_bound_for_equivalence_class}
    \mathbb{P}(M_n^k\geq \ell) \leq Ce^{-c\ell^{1/128}}.
\end{equation}
\end{proposition}
Finally we have the following lemma about the size of intersection of geodesics with different directions. Recall the line segments $\mathcal{J}'_i, \mathcal{J}_i'', \mathcal{J}_j',\mathcal{J}_j''$ and the parallelogram $W_{i,j}$ in the proof of Theorem \ref{depth_bounds_theorem} lower bound. We have the following lemma.
\begin{lemma}\cite[Lemma 3.2, Lemma 3.3]{BBB23}
\label{lemma: 3.2, 3.3}
Consider any geodesic $\Gamma_i$ (resp.\ $\Gamma_j$) starting from $\mathcal{J}_i'$ (resp.\ $\mathcal{J}_j'$) and ending at $\mathcal{J}_i''$ (resp.\ $\mathcal{J}_j'')$. If $I_{\Gamma_i,\Gamma_j}$ denote the intersection size of $\Gamma_i$ and $\Gamma_j$ inside $W_{i,j}$ then for $\ell$ and $n$ large enough there exist $C,c>0$ such that we have 
\begin{displaymath}
    \mathbb{P}\left(\max I_{\Gamma_i,\Gamma_j} \geq \frac{\ell \log|k_{i,j}| n}{k_{i,j}^3}\right) \leq Ce^{-c \ell}.
\end{displaymath}
\end{lemma}
    
\printbibliography
\end{document}